\documentclass{article}
\usepackage{amsmath,amssymb,latexsym,amsthm}

\newtheorem{theo}{Theorem}[section]
\newtheorem{pro}[theo]{Proposition}
\newtheorem{lem}[theo]{Lemma}
\newtheorem{cor}[theo]{Corollary}
\newtheorem{claim}[theo]{Claim}

\theoremstyle{definition}
\newtheorem{defin}[theo]{Definition}

\theoremstyle{remark}
\newtheorem{rem}[theo]{Remark}

\newcommand{\ra}{\rightarrow}

\title{Refined estimates for some basic random walks on the
symmetric and alternating groups}
\author{L. Saloff-Coste\thanks{Research partially supported by NSF
grant DMS 0603806}\\
{\small Department of mathematics}\\
{\small Cornell University}
\and  J. Z\'u\~niga\thanks{Research partially supported by NSF
grant DMS 0603806 and DMS 0306194}\\
{\small Department of mathematics}\\
{\small Cornell University}}

\begin{document}
\maketitle

\begin{abstract}
We give refined estimates for the discrete time and continuous 
time versions of some basic random walks on the symmetric and alternating groups $S_n$ and 
$A_n$.  We consider the following models:
random transposition, transpose top with random, random insertion, and walks 
generated by the uniform measure on a conjugacy class.  
In the case of random walks on $S_n$ and $A_n$ generated by the uniform measure 
on a conjugacy class, we show that in continuous time 
the $\ell^2$-cuttoff has a lower bound of $(n/2)\log n$.  This result, along with the
results of M\"uller, Schlage-Puchta and Roichman, demonstrates that the 
continuous time version of these walks may take much longer to reach stationarity  
than its discrete time counterpart. 
\end{abstract}

\section{Introduction}

This work is concerned with some basic random walks on the symmetric
group, $S_n$, and the alternating group, $A_n$. 
Specifically, we are interested in
the following models: (a) Random transposition and transpose top with random;
(b) walks generated by the uniform measure on a conjugacy class, e.g.,
$4$-cycles or $k_n$-cycles with $k_n$ an increasing function of $n$;
(c) random insertion. Although these walks have been studied extensively,
we obtain here results that either improved upon known estimates
or complement those estimates.

The convergence of the random transposition walk on $S_n$
was studied by Diaconis and Shahshahani in \cite{DS}.
We present a technical improvement  of their fundamental result.
This is motivated by the role played by random transposition
in the comparison technique of \cite{DSC}: any improvement upon the 
$\ell^2$ convergence of the random transposition walk  has consequences 
for a wealth of other walks. We will illustrate this by
obtaining the best known result for the random insertion walk. These results
are also used in \cite{SCZ} to study certain time in-homogeneous versions
of the random insertion walk and this was indeed
our original motivation for developing the results presented here.
For an overview of results connected to the random transposition walk, 
see \cite{Dtran}.

The transpose top with random walk is an interesting example mentioned in
\cite{FOW} and in \cite{Dia2} but details of its $\ell^2$
analysis have never appeared in print. (This walk should not be confused 
with the more classical top to random walk studied in \cite{DFP}.)
The estimates concerning this walk that are 
proved here are used in \cite{SCZ} to obtain the best known convergence bounds for a class
of time  in-homogeneous processes called semi-random transpositions.

Random walks associated with conjugacy classes other than
the class of transpositions have been studied by
\cite{MSP,SP,LP,Roi,Ro}. For most of those walks, we show
that $\ell^2$ convergence occurs at very different times for the
discrete time process and the continuous time process.
Although this phenomenon is simple to understand a posteriori, it is
a bit surprising at first and is often overlooked.

Let us briefly describe our notation. On a finite group $G$
with identity element $e$, the
random walk started at $e$ driven by a given probability measure $q$
is the process $X_n=\xi_1\cdot \dots \cdot \xi_n$
where the $\xi_i$ are independent $G$-valued random variables
with distribution $q$.
The distribution of $X_n$ is $q^{(n)}$, the convolution of $q$
with itself, $n$ times. Any such walk admits the uniform measure
$u$ as an invariant measure. It is reversible if and only
if $q(x)=q(x^{-1})$ for all $x$. The walks studied here
all have this property. We are mostly interested in the quantity
($\chi$-square distance)
$$d_2(q^{(n)},u)=\left(|G|\sum_{G}|q^{(n)}-u|^2\right)^{1/2},\;\;
u\equiv 1/|G|.$$
This is always an upper bound for
$2\|q^{(n)}-u\|_{\mbox{\tiny TV}}$
where
$$\|q-p\|_{\mbox{\tiny TV}}=\sup_A\{q(A)-p(A)\}$$ is the total
variation distance between the probability measures $p$ and $q$.

Given such a discrete time process, we also
consider the associated continuous time process whose distribution at time
$t\in [0,\infty)$ is given by
\begin{equation*}
h_t(x)=h_{q,t}(x)
=e^{-t}\sum_{n=0}^{\infty}\frac{t^n}{n!}q^{(n)}(x).
\end{equation*}

We now state some of the results proved in this work.
Random transposition is the walk on the symmetric group $G=S_n$
driven by $q=q_{\mbox{\tiny RT}}$
where
$$q_{\mbox{\tiny RT}}(\tau)=\left\{\begin{array}{ll}
2/n^2 & \textrm{if $\tau=(i,j)$, $1\leq i,j\leq n, \;\; i\neq j,$}\\
1/n & \textrm{if $\tau=e$}\\
0 & \textrm{otherwise.}
\end{array}\right.
$$

\begin{theo}
Let $q$ be the random transposition measure on the group $S_n$, $n>14$.
For any $c\geq 0$ and $t\geq \frac{n}{2}(\log{n}+c)$, we have
$
d_2(q^{(t)},u)\leq 2 e^{-c}.$
\end{theo}
In \cite{DS}, Diaconis and Shahshahani proved this result
with an unspecified constant $B$ instead of $2$ in front of $e^{-c}$
and for large enough $n$. In this paper their approach is refined to obtained
the bound stated above. We also prove a similar result in continuous time
which turns out to be somewhat more difficult.
Having good control of $d_2(h_{q_{\mbox{\tiny RT}},t},u)$ is very useful
in connection with the comparison techniques of \cite{DSC}. See Section \ref{sec-RI} 
where this is used to study the random insertion walk.

Transpose top with random is the process driven by
$q(\tau)=1/n$ if $\tau\in \{(1,i), i=1,\dots,n\}$ (where $(1,1)=e$)
and $0$ otherwise.
\begin{theo}
Let $q$ be the transpose top with random  measure on the group $S_n$.
For any $c\geq 0$ and $t\geq n(\log{n}+c)$, we have
$
d_2(q^{(t)},u)\leq \sqrt{2} e^{-c}.$
\end{theo}

To illustrate our results concerning walks driven by  conjugacy classes,
consider the measure
$q_{\mathbf c_n}$ which, for each $n$,  is uniform on
$\mathbf c_n\subset S_n$, the conjugacy class
of all cycles of odd length $k_n=2m_n+1$. The corresponding walk is on $A_n$.
\begin{theo} Fix $\epsilon\in(0,1)$ and set
$t_n=\frac{n}{2}\log n$.
Referring to the continuous time process with distribution
$h_{\mathbf c_n, t}=h_{q_{\mathbf c_n}, t}$
associated to the cycle walk on $A_n$ described above,
if $m_n$ tends to infinity with $n$, we have (with $u_n\equiv1/|A_n|=2/n!$)
$$\lim_{n\rightarrow \infty}d_2(h_{\mathbf c_n,(1+\epsilon)t_n},u_n)=0
\mbox{ and } 
\lim_{n\rightarrow \infty}d_2(h_{\mathbf c_n, (1-\epsilon)t_n},u_n)=\infty.$$\end{theo}
This result shows an $\ell^2$-cutoff at time $(n/2)\log n$.
When $k_n< cn$ for some $c\in (0,1)$,
Roichman \cite{Roi} shows that the associated discrete time process
has a mixing time in $\ell^2$ of order $(n/k_n)\log n$.
Roichman's results are improved in \cite{MSP,SP}.
As $k_m=2m_n+1\rightarrow \infty$, the discrete mixing time
$(n/k_n)\log n$ is much smaller than the continuous
cutoff time $(n/2) \log n$. The explanation is simple. Consider the 
eigenvalues of the walk driven by $q_{\mathbf c_n}$, that is, the eigenvalues
of the convolution operator $f\ra f*q_{\mathbf c_n}:\ell^2(u_n)\ra\ell^2(u_n)$, call these 
eigenvalues $\alpha_i$.  In continuous time, the $\ell^2$ cutoff time is controlled 
by the very large number of very small eigenvalues. These small eigenvalues
contribute significantly in continious time because they apear in the 
form $e^{-t(1-\alpha_i)}$.
In discrete time, these small eigenvalues do not contribute much since they appear
in the form $\alpha_i^t$.
Although the explanation is simple, verifying that this is indeed
the case is not an easy task. We will prove similar results 
for general conjugacy classes.

\section{Review and notation}
We refer the reader to \cite{Dia,S-K} for careful introduction
to random walks on finite groups. We briefly review some of the needed
material below.
\subsection{Cutoffs}
Many examples of random walks on groups
that have been studied demonstrate a unique behavior called the
cutoff phenomenon. This was first studied in \cite{Al,AD,DS}.
See also \cite{Dcut,GYSC,StF,S-K}.
\begin{defin}
Let $(G_n)_0^{\infty}$ be a sequence of finite groups and denote by
$u_n$ the uniform measure on $G_n$.
For each $n\geq 0$ consider the random
walk on $G_n$ driven by the measure $q_n$.
The sequence $((G_n,q_n))_0^{\infty}$
is said to have total variation cutoff (resp. $\ell^2$)
if there is a sequence $(t_n)_0^{\infty}$
with $t_n\rightarrow \infty$ such that for any $\epsilon\in(0,1)$
\begin{itemize}
\item[(1)] if $k_n=(1+\epsilon)t_n$ then
$d_{\mbox{\tiny{\em TV}}}(p_n^{(k_n)},u_n)\ra 0$  (resp.
$d_2(p_n^{(k_n)},u_n)\ra 0$);
\item[(2)] if $k_n=(1-\epsilon)t_n$ then
$d_{\mbox{\tiny{\em TV}}}(p_n^{(k_n)},u_n)\ra 1$ (resp.
$d_2(p_n^{(k_n)},u_n)\ra \infty$).
\end{itemize}
The sequence $((G_n,q_n))_0^{\infty}$ is said to demonstrate a total variation
(resp. $\ell^2$) pre-cutoff if there exist constants $0<a<b$ such that
\begin{itemize}
\item[(1)] $\liminf_{n\ra\infty}d_{\mbox{\tiny{\em TV}}}(p_n^{(at_n)},u_n)> 0$
(resp. $\liminf_{n\ra\infty}d_2(p_n^{(at_n)},u_n)> 0$);
\item[(2)] $\lim_{n\ra\infty}d_{\mbox{\tiny{\em TV}}}(p_n^{(bt_n)},u_n)=0$
(resp. $\liminf_{n\ra\infty}d_2(p_n^{(bt_n)},u_n)=0$).
\end{itemize}
\end{defin}
Similar definitions apply in continuous time. Diaconis and Shahshahani
proved in \cite{DS} that the random transposition walk on $S_n$ has a cutoff
(both in total variation and $\ell^2$) at time $(n/2)\log n$.
For a overview of other results in this direction, see \cite{Dcut,S-K}.

\subsection{Eigenvalues and representation theory}\label{sec-RepT}
It is well known that for reversible finite Markov chains, the
$\chi$-square distance can be expressed in terms of eigenvalues
and eigenfunctions. See, e.g., \cite{StF}. For a reversible random walk
on a finite group $G$ driven by $q$,
the expression simplifies and the eigenvectors drop out.
If we let $\beta_i$, $i=0,\dots,|G|-1$, be the eigenvalues
of the operator of convolution by $q$ acting on $\ell^2(G)$,
in non-increasing order and repeated according to multiplicity, we have
\begin{equation}\label{dist-spect}
d_2(q^{(t)},u)^2=\sum_{i=1}^{|G|-1}\beta_i^{2t}
\;\;\text{and}\;\;
d_2(h_t,u)^2=\sum_{i=1}^{|G|-1}e^{-2t(1-\beta_i)}.
\end{equation}

Representation theory provides a tool that can be helpful to compute
eigenvalues. We give a very brief review of these methods.
All the material in this section can be found in greater detail
in \cite{Dia,Sa}.
A {\em representation} of a finite group $G$ on a vector space $V$ is
a homomorphism  $\rho:G\ra GL(V)$ where $GL(V)$ is the group of general linear
transformations of $V$.  We say that
$\rho$ has dimension $\mbox{d}_{\rho}$ where $\mbox{d}_{\rho}$ is equal to the
dimension of $V$.  Let $W\subset V$, if $\rho W=W$ then $\rho|_W$  is called
a {\em subrepresentation} of $\rho$.
A representation $\rho$ is called {\em irreducible}
if it admits no nontrivial subrepresentation.
The {\em character} of a representation
$\rho$ at $s\in G$ is $\chi_{\rho}=\mbox{Tr}(\rho(s))$.
Characters are constant under conjugation, i.e.
for any $x,y\in G$ then
\begin{equation*}
\chi_{\rho}(x^{-1}yx)=\chi_{\rho}(y).
\end{equation*}

For $f:G\ra\mathbb{R}$, the {\em Fourier transform}
of $f$ at $\rho$ is
\begin{equation*}
\widehat{f}(\rho)=\sum_{s\in G}f(s)\rho(s).
\end{equation*}
The Fourier transform converts convolution of functions into multiplication
of matrices (or composition of linear maps)
$\widehat{f*g}(\rho)=\widehat{f}(\rho)\widehat{g}(\rho)$.
If $G$ is a finite group and if $f,g$ are any two functions taking values
on $G$ then the {\em Plancharel formula}
relates the convolution of $f$ and $g$ at $e$ to
the Fourier transform as follows
\begin{equation*}
f*g(e)=\sum_{s\in G}f(s^{-1})g(s)=\frac{1}{|G|}
\sum_{\rho}\mbox{d}_{\rho}\mbox{Tr}(\widehat{f}(\rho)\widehat{g}(\rho))
\end{equation*}
\noindent where $|G|$ is the order of $G$ and the sum is over all
(equivalent classes of) irreducible representations of $G$.
In what follows $\rho\neq 1$ means that $\rho$ is not the trivial representation.
The Plancharel formula is used to obtain the following proposition.

\begin{pro}\label{dist-trace}
Let $G$ be a finite group equipped with a probability measure $q$
satisfying $q(x)=q(x^{-1})$, $x\in G$. We have
\begin{equation}\label{eq-disc-trace}
d_2(q^{(t)},u)^2=\sum_{\rho\neq 1}\mbox{\em d}_{\rho}\mbox{\em Tr}(\widehat{q}(\rho)^{2t}).
\end{equation}
\end{pro}

In general, it is very difficult to  estimate
$\mbox{Tr}(\widehat{q}(\rho)^{t})$. However, in the
case were $q$ is a class function, i.e.,
$q(x^{-1}yx)=q(y)$, for all $x,y\in G$,
a celebrated lemma of Schur provides a nice analysis.
If $\rho$ is an irreducible representation and $(\mathcal{C}_j)_1^m$ are the
conjugacy classes of the group $G$ then  $\widehat{q}(\rho)$ is a constant
multiple of
the identity matrix. This yields
\begin{equation*}
\widehat{q}(\rho)=I_{\mbox{\tiny d}_\rho} \cdot
\left(\sum_{j=1}^mq(\mathcal{C}_j)\frac{\chi_{\rho}(c_j)}{\mbox{d}_{\rho}}\right)
\end{equation*}
\noindent where $c_j\in \mathcal{C}_j$.
For a proof of this fact see \cite{Dia,Dia2}. The next proposition now follows

\begin{pro}\label{dist-character}
Let $G$ be a finite group and $q$ a probability measure on $G$
satisfying $q(x^{-1})=q(x)$, $x\in G$.
If $q$ is constant on conjugacy classes then
\begin{eqnarray}
d_2(q^{(t)},u)^2
&=&\sum_{\rho\neq 1}\mbox{\em d}_{\rho}^2\left(\sum_{j=1}^mq(\mathcal{C}_j)\frac{\chi_{\rho}(c_i)}{\mbox{\em d}_{\rho}}\right)^{2t}
\text{ and}\nonumber\\
d_2(h_{\mathcal C,t},u)^2
&=&\sum_{\rho\neq 1}\mbox{\em d}_{\rho}^2\exp\left(-2t\left( 1-\sum_{j=1}^mq(\mathcal C_j)\frac{\chi_{\rho}(c_i)}{\mbox{\em d}_{\rho}}\right)\right)\label{main=}.
\end{eqnarray}
\end{pro}

To connect more directly representation theory with the usual spectral 
decomposition, let $\rho:G\ra GL(V)$ a representation of $G$ on a finite vector space 
$V$ equipped with an invariant Hermitian product $\langle\cdot,\cdot\rangle$.
Fix a probability measure $q$ and consider the linear transformation
$\widehat{q}(\rho): V\ra V$.
Suppose  $e_i,e_j$ are unit vectors in $V$ and that $e_i$ is an eigenvector
of $\widehat{q}(\rho)$ with eigenvalue $\gamma_i$.
Set $\phi_{i,j,\rho}(x)=\langle\rho(x)e_i,e_j\rangle$.
We claim that $\phi_{i,j,\rho}$ is an eigenfunction for 
$f\mapsto f*q$ with
\begin{equation*}
f*q(x)=\sum_yf(xy^{-1})q(y)
\end{equation*}
on $\ell^2(G)$ with eigenvalue $\gamma_j$. Indeed,
\begin{eqnarray*}
\phi_{i,j,\rho}*q(x)
&=&\sum_yq(y)\langle\rho(xy^{-1})e_i,e_j\rangle
=\left\langle\rho(x)e_i,\sum_yq(y)\rho(y)e_j\right\rangle\\
&=&\langle\rho(x)e_i,\widehat{q}(\rho)e_j\rangle
= \gamma_j\langle \rho(x)e_i,e_j\rangle= \gamma_j\phi_{i,j,\rho}(x).
\end{eqnarray*}
Now, if $q$ is symmetric and thus  $\widehat{q}(\rho)$ is diagonalizable in an orthonormal 
basis $(e_i)_1^{\mbox{\tiny d}_\rho}$ then the construction above yields $\mbox{d}_\rho$ 
eigenvalues and $\mbox{d}_\rho^2$ orthonormal eigenvectors in $\ell^2(G)$,
each eigenvalue having multiplicity $\mbox{d}_\rho$. Furthermore, if $\rho$, $\rho'$, 
are two inequivalent irreducible representations the corresponding eigenvectors are 
orthogonal (some of the eigenvalues may be the same). A proof of the orthogonality of 
$\phi_{i,j,\rho}$ is given in Corollary 4.10 of \cite{Knapp}.
Hence, this produces  $|G|$ orthonormal eigenfunctions 
since  $\sum_{\rho}d_{\rho}^2=|G|$ where the sum is
taken over all (equivalent classes of) irreducible representations.

For future reference we mention the well known fact that
irreducible representations on $S_n$ are indexed by the Young diagram
with $n$ boxes (see \cite{Sa}).

\begin{defin}
Let $\lambda=(\lambda_1,\dots,\lambda_m)$ be a partition of\; $n$ so that
$\lambda_1\geq\lambda_2\geq\dots\geq\lambda_m$ and $\sum_{i=1}^m\lambda_i=n$.
$\lambda$ is called a Young diagram of $n$ boxes and $\lambda_i$ denotes the
number of boxes in the $i$-th row of the diagram.  
\end{defin}

\begin{figure}[h]

\begin{center}
\caption{The Young diagram for $\lambda=(5,4,2,1)$} \vspace{.05in}
\end{center}

\begin{picture}(300,100)(-100,0)

\put(0,100){\line(1,0){100}}
\put(0,80){\line(1,0){100}}
\put(0,60){\line(1,0){80}}
\put(0,40){\line(1,0){40}}
\put(0,20){\line(1,0){20}}
\put(0,100){\line(0,-1){80}}
\put(20,100){\line(0,-1){80}}
\put(40,100){\line(0,-1){60}}
\put(60,100){\line(0,-1){40}}
\put(80,100){\line(0,-1){40}}
\put(100,100){\line(0,-1){20}}
\end{picture}

\end{figure}
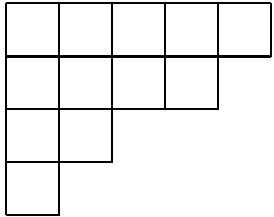

\noindent The association of an irreducible representation to a Young diagram
will provide a key tool to calculate the normalized character
${\chi_{\rho}(\cdot)}/{\mbox{\em d}_{\rho}}$ of an irreducible representation $\rho$ and the eigenvalues of many of the walks we study.
This technique is illustrated in the following sections.

%%%%%%%%%%%%%%%%%%%%%%%%%%%%%%%%%%%%%%%%%%%%%%%%%%%%%%%%%%%%%%%%%%%%%%%%
%%%%%%%%%%%%%%%%%%%%%%%%%%%%%%%%%%%%%%%%%%%%%%%%%%%%%%%%%%%%%%%%%%%%%%%%
\section{Transpose top with random}

Consider the following shuffling method of a deck of $n$ cards:
pick a card uniformly
at random from the deck and transpose it with the top card. This shuffling
scheme is described by the measure $q$ on the symmetric group $G=S_n$ where
\begin{displaymath}
q(\tau)=\left\{\begin{array}{ll}
1/n & \textrm{if $\tau=(1,j)$, $1\leq j\leq n$}\\
0 & \textrm{otherwise.}
\end{array}\right.
\end{displaymath}
This walk is called transpose top with random.

In order to establish an upper bound for the $\ell^2$ mixing time,
the tools from  group representation presented in section \ref{sec-RepT}
are used to calculate the eigenvalues of $q$.  Most of the needed computations are in 
\cite{FOW} and the procedure is outlined in \cite{Dia2} where it is stated that transpose 
top with random has a cutoff time of $n\log n$. The following theorem gives a more 
precise upper bound. This result is used in \cite{SCZ} to study a class of time 
inhomogeneous chains called semi-random transpositions.

\begin{theo}\label{thm-tr-disc}
Let $q$ be the transpose top with random measure on the group $S_n$.
If $n\geq 1$, $c\geq 0$, and $t\geq n(\log{n}+c)$
\begin{equation*}
d_{2}(q^{(t)},u)\leq\sqrt{2}\;e^{-c},\;\;\;
d_{2}(h_{q,t},u)\leq\sqrt{2}\;e^{-c}.
\end{equation*}
\end{theo}
\begin{proof}
By Proposition \ref{dist-trace},
\begin{equation*}
d_2(q^{(t)},u)^2=\sum_{\rho\neq 1}\mbox{d}_{\rho}\mbox{Tr}(\widehat{q}(\rho)^{2t}).
\end{equation*}
\noindent Even though $q$ is not constant on conjugacy classes \cite{Dia2}
notes that $q$ is invariant under conjugation by elements of $S_{n-1}$ where
\begin{equation*}
S_{n-1}=\left\{ \tau\in S_{n}|\tau(1)=1\right\}.
\end{equation*}
\noindent Using this fact it is shown that $\widehat{q}(\rho)$ is a diagonal matrix (with real entries). 
See \cite{FOW,Dia2}. Therefore
\begin{equation}\label{tr-bound1}
\mbox{Tr}\left(\widehat{q}(\rho)^{2t}\right)=
\sum_{i=1}^{{\mbox{\tiny{d}}}_{\rho}}\alpha_i^{2t}\leq \mbox{d}_{\rho}\alpha_1^{2t}
\end{equation}
\noindent where $\alpha_1\geq\cdots\geq \alpha_{{\mbox{\tiny{d}}}_{\rho}}$
are the eigenvalues of $\widehat{q}(\rho)$.

To compute $\alpha_i$, consider $M=\sum_{i=2}^n\rho((1, i))$.
Let $\lambda=(\lambda_1,\dots,\lambda_m)$ be the Young diagram associated
to the irreducible representation $\rho$.  Let
$\sigma_1\leq \cdots\leq \sigma_{\mbox{\tiny{d}}_{\rho}}$
be the eigenvalues of $M$.
In \cite{FOW} it is shown
that for $1\leq i\leq \mbox{d}_{\rho}$ then
\begin{equation*}
\sigma_i=\lambda_i-i.
\end{equation*}
The multiplicity of each $\sigma_i$ is also described in \cite{FOW}.
We do not need this for the present proof but, to give an example,
if $\lambda=(n-1,1)$ then  the eigenvalues of $M$ are $\sigma_1= n-2$ with multiplicity 
$n-2$ and $\sigma_2= -1$ with multiplicity $1$.

\noindent As
\begin{equation*}
\widehat{q}(\rho)
=\sum_{\tau\in G}q(\tau)\rho(\tau)=\sum_{i=1}^n\left(\frac{1}{n}\right)\rho((1, i))
=\frac{M+\rho(e)}{n}=\frac{M+I}{n}
\end{equation*}
\noindent where $I$ is the identity matrix of dimension $\mbox{d}_{\rho}$,
we easily obtain the eigenvalues $\alpha_i$, $1\leq i\leq\mbox{d}_{\rho}$:
$\alpha_i=(\sigma_i +1)/n$. In particular, $\alpha_1=\lambda_1/n$.

Denote by $\rho_{\lambda}$ the irreducible representation associated
to a partition $\lambda$ and by $\rho_{\lambda}=1$ the trivial representation
with corresponds to $\lambda=(n)$. Equation (\ref{tr-bound1}) yields
\begin{equation}\label{tr-bound2}
d_2(q^{(t)},u)^2\leq \sum_{\rho_{\lambda}\neq 1}
     \mbox{d}_{\rho_\lambda}^2\left(\frac{\lambda_1}{n}\right)^{2t}
=\sum_{j=1}^{n-1}\sum_{{\rho_{\lambda}}\atop{\lambda_1=n-j}}
     d_{\lambda}^2\left(\frac{\lambda_1}{n}\right)^{2t}.
\end{equation}

\noindent In \cite{Dia, FOW} it is shown that for $l\geq 1$
\begin{equation}
\sum_{{\rho_{\lambda}}\atop{\lambda_1=l}}d_{\rho_{\lambda}}^2\leq {n\choose{l}}^2(n-l)!.\label{dim}
\end{equation}
It follows that for $c\geq 0$ and $t\geq n(\log{n}+c)$
\begin{eqnarray*}
d_2(q^{(t)},u)^2&\leq&
\sum_{j=1}^{n-1}\left(\frac{n!}{(n-j)!}\right)^2\left(\frac{1}{j!}\right)\left(1-\frac{j}{n}\right)^{2t}\\
&\leq& \sum_{j=1}^{n-1} n^{2j}\left(\frac{1}{j!}\right)e^{-2j\log{n}}e^{-2jc}=(e-1)e^{-2c}\leq 2e^{-2c}.
\end{eqnarray*}
For the continuous time process, we have, similarly,
\begin{eqnarray*}
d_2(h_{q,t},u)^2&\le & \sum_{\rho_{\lambda}\neq 1}
     \mbox{d}_{\rho_\lambda}^2 e^{-2t(1-\alpha_1)}
=\sum_{j=1}^{n-1}\sum_{{\rho_{\lambda}}\atop{\lambda_1=n-j}}
     d_{\lambda}^2\exp\{-2t(1-{\lambda_1}/n)\}\\
&=&\sum_{j=1}^{n-1}\sum_{{\rho_{\lambda}}\atop{\lambda_1=n-j}}\mbox{d}_{\lambda}^2e^{-2tj/n}
\leq\sum_{j=1}^{n-1}\left(\frac{n!}{(n-j)!}\right)^2\left(\frac{1}{j!}\right)e^{-2tj/n}
\end{eqnarray*}
\noindent where the last inequality follows from (\ref{dim}).
Again, if $n\geq 1$, $c\geq 0$ and $t\geq n(\log{n}+c)$ then
\begin{equation*}
d_2(h_{q,t},u)^2\leq\sum_{j=1}^{n-1} n^{2j}\left(\frac{1}{j!}\right)e^{-2j\log{n}}e^{-2jc}\leq 2e^{-2c}.
\end{equation*}
\end{proof}

The next proposition shows that transpose top with random
 has  $\ell^2$ and  total variation cutoffs at time $n\log{n}$.

\begin{pro}\label{thm-tr-lower}
Let $q$ be the transpose top with random measure on $S_n$.  For any
sequence $(k_n)_0^{\infty}$ such that $(k_n-n\log{n})/n$ tends to $-\infty$ as $n$ tends to $\infty$ then
\begin{equation*}
d_2(q^{(k_n)},u)\ra\infty \;\; \text{and} \;\; d_{\mbox{\tiny{\em TV}}}(q^{(k_n)},u)\ra 1.
\end{equation*}
\end{pro}

\begin{proof} For the $\ell^2$ bound, we observe that \cite{FOW}
also gives a description of the multiplicity of the eigenvalues.
In particular, if $\lambda=(n-1,1)$ then the eigenvalue $1-1/n$ of $\widehat{q}(\rho_\lambda)$ 
has multiplicity $n-2$. Since $d_\lambda=n-1$ we get that 
$$d_2(q^{(k)},u)^2\ge  (n-1)(n-2)(1-1/n)^{2k}$$
from which the desired $\ell^2$ statement easily follows.

\begin{rem}
Let $\varphi(\sigma)$ be the number of fixed points of $\sigma$.
One can check by direct inspection that
\begin{eqnarray*}
f(\sigma)=\left(\frac{n-1}{n-2}\right)^{1/2}\times\left\{\begin{array}{ll}
\varphi(\sigma)  -2 & \textrm{if $\sigma(1)=1$}\\
\varphi(\sigma) -1+\frac{1}{n-1}  & \textrm{if $\sigma(1)\neq 1$}
\end{array}\right.
\end{eqnarray*}
is a normalized eigenfunction (for convolution by $q$)
with eigenvalue $1-1/n$.
Its value at $e$ is $f(e)^2=(n-1)(n-2)$.
This gives a entirely elementary proof of the $\ell^2$ lower bound
since $d_2(q^{(k)},u)^2\ge (1-1/n)^{2k}f(e)^2.$  The previouse inequality  
results from the fact that one can write the $\chi$-square distance 
in terms of eigenvalues and eigenfuctions.  See, e.g., \cite{StF}.
\end{rem}

The proof of the lower bound for total variation follows mostly
an argument used in
\cite{AD} to give a lower bound for random transposition (and for
the top to random insertion shuffle).
Let 
\begin{equation}\label{eq-Aj-lower}
A_j=\left\{\sigma\in S_n:\varphi(\sigma)\geq j\right\}
\end{equation}
with $\varphi$ as defined above. Then
\begin{equation*}
d_{\mbox{\tiny TV}}(q^{(k_n)},u)\geq q^{(k_n)}(A_j)-u(A_j).
\end{equation*}
\noindent Calculating $u(A_j)$ is equivalent to calculating the
probability of at least $j$ matches in the classical matching problem.
In \cite{Fe}, Feller gives a closed form solution for $u(A_j)$.
Using this we get the following estimate for $j\geq 2$
\begin{equation}\label{uniform}
u(A_j)=\sum_{m=j}^n \frac{1}{m!}\left(\sum_{v=0}^{n-m}\frac{(-1)^v}{v!}\right)\leq e^{-1}\left(\frac{1}{(j-1)!}\right).
\end{equation}
Next we bound $q^{(k_n)}(A_j)$ from below.
Consider the experiment where successive balls are droped 
independently and uniformly at random into $n$ boxes. Let 
$B_{j,k}$ be the event that after dropping $k$ balls
there are at least $j$ empty boxes.  Then
\begin{equation*}
q^{(k_n)}(A_{j-1})\geq P(B_{j,k_n}).
\end{equation*}
Let  $V_l$ be the number of balls dropped when exactly $l$ boxes are filled.  
We have
\begin{equation*}
P(B_{j,k_n})= P(V_{n-j}\geq k_n)\geq 1- P(V_{n-j}\leq k_n).
\end{equation*}
We would like to show that for any fixed  $j$,
$P(V_{n-j}\leq k_n)\ra 0$ as $n\ra \infty$.
We have
\begin{equation*}
V_{n-j}=(V_{n-j}-V_{n-j-1})+(V_{n-j-1}-V_{n-j-2})+\cdots+(V_2-V_1)+V_1.
\end{equation*}
The $V_{i+1}-V_i$ are independent random variables with geometric distribution
\begin{equation*}
P\left\{V_{i+1}-V_i=l\right\}=\left(\frac{n-i}{n}\right)
\left(1-\frac{n-i}{n}\right)^{l-1},\;\; l\ge 1.
\end{equation*}
\noindent Hence
\begin{equation*}
E(V_{i+1}-V_i)=\frac{n}{n-i} \;\;\text{and}\;\; \text{Var}(V_{i+1}-V_i)=\left(\frac{n}{n-i}\right)^2\left(1-\frac{n-i}{n}\right).
\end{equation*}

\noindent It follows that
\begin{eqnarray*}
E(V_{n-j})&=&\sum_{i=1}^{n-j-1}\frac{n}{n-i}\geq\int_{0}^{n-j-1}\frac{n}{n-x}\;dx \geq n\log\left(\frac{n}{j+1}\right)
\end{eqnarray*}
\noindent and
\begin{eqnarray*}
\text{Var}(V_{n-j})&=&\sum_{i=1}^{n-j-1} \frac{n^2}{(n-i)^2}-\frac{n^2}{n(n-i)} \leq \sum_{i=1}^{n-j-1}\frac{n^2}{(n-i)^2}\\
&\leq& \int_1^{n-j}\left(\frac{n}{n-x}\right)^2\; dx \leq \frac{n^2}{j}.
\end{eqnarray*}
By assumption $k_n=n\log{n}-nc_n$ and $c_n\ra\infty$ as $n\ra\infty$.
If we assume, as we may,
that $c_n>\log(j+1)$ then Chebyshev's inequality gives
\begin{eqnarray*}
P(V_{n-j}\leq k_n)&=&P(V_{n-j}\leq n\log{n}-nc_n)\\
&\leq&P(n(c_n-\log(j+1))\leq |(V_{n-j})-E(V_{n-j})|)\\
&\leq&\frac{\text{Var}(V_{n-j})}{n^2(c_n-\log(j+1))^2}
\leq\frac{1}{j(c_n-\log(j+1))^2}.
\end{eqnarray*}

\noindent This yields
\begin{equation*}
\lim_{
n\ra \infty} d_{\mbox{\tiny TV}}(q^{(k_n)},u)\geq 
\lim_{
n\ra \infty} (P(B_{j+1,k_n})-u(A_j))\geq 
1-e^{-1}\left(\frac{1}{(j-1)!}\right).
\end{equation*}
Since $j$ is arbitrary the desired result follows.
\end{proof}

\begin{cor}
Let $h_t$ be the distribution for the continuous time process associated
to the transpose top with random measure $q$.  For any
sequence $(k_n)_0^{\infty}$ such that $(k_n-n\log{n})/n$ tends to $-\infty$ as
$n$ tends to $\infty$ then
\begin{equation*}
d_2(h_{k_n},u)\ra\infty \;\; \text{and} \;\; d_{\mbox{\tiny{\em TV}}}(h_{k_n},u)\ra 1.
\end{equation*}
\end{cor}
\begin{proof} The $\ell^2$ bound follows from the same argument used above.
In the case of the total variation bound, one can show that 
for $A_j$ defined in (\ref{eq-Aj-lower}) then
$h_{k_n}(A_j)\ra 1$.  A sight modification of the proof 
of Proposition \ref{thm-tr-lower} gives that for $\alpha\in(1/2,1)$ 
\begin{equation*}
\lim_{n\ra\infty}q^{k_n+k_n^{\alpha}}(A_j)= 1.
\end{equation*}
Combining the limit above wih the fact that 
\begin{equation*}
\lim_{n\ra\infty}\sum_{t=0}^{k_n+k_n^{\alpha}}e^{-k_n}\frac{k_n^t}{t!}
=\lim_{n\ra\infty}P\left(\frac{X_n-k_n}{\sqrt{k_n}}\leq k_n^{\alpha-1/2}\right)
=1\end{equation*}
where $X_n$ is a Poisson random variable with parameter $k_n$ gives us the 
desired result. 
\end{proof}

%%%%%%%%%%%%%%%%%%%%%%%%%%%%%%%%%%%%%%%%%%%%%%%%%%%%%%%%%%%
%%%%%%%%%%%%%%%%%%%%%%%%%%%%%%%%%%%%%%%%%%%%%%%%%%%%%%%%%%%
\section{Random transpositions}\label{sec-RandomTrans}
\subsection{Discrete time}
Consider the following measure $q=q_{\mbox{\tiny RT}}$ on the group $G=S_n$,
\begin{eqnarray}\label{rt-measure}
q(\tau)=\left\{\begin{array}{ll}
2/n^2 & \textrm{if $\tau=(i,j)$, $1\leq i,j\leq n, \;\; i\neq j,$}\\
1/n & \textrm{if $\tau=id,$}\\
0 & \textrm{otherwise.}
\end{array}\right.
\end{eqnarray}

\noindent The measure $q$ models the shuffle of a deck of $n$ cards
where one picks two cards independently and uniformly at random and
transposes them.  The random transposition shuffle has
been shown to demonstrate cutoff at $(n/2)\log{n}$, see \cite{Dia, Dia2,DS}.

\begin{theo} {\bf (Diaconis and Shahshahani)}\label{thm-DS}
Let $q$ be the random transposition measure on the group $S_n$
then there exists a positive universal constant $B$ such that for
any $c\geq 0$ and $t\geq \frac{n}{2}(\log{n}+c)$ then
\begin{displaymath}\label{rt-bound}
2d_{\mbox{\tiny{\em{TV}}}}(q^{(t)},u)\leq d_2(q^{(t)},u)\leq Be^{-c}.
\end{displaymath}
\end{theo}
\noindent One of the aims of this section is to get a more
precise estimate on the constant $B$ in the theorem above.

\begin{pro}\label{rt}
Let $q$ be the random transposition measure on $S_n$.
For $n\geq 14$, $c\geq 0$, and  $t\geq \frac{n}{2}(\log{n}+c)$
then equation (\ref{rt-bound}) holds with
\begin{equation*}
B^2\leq 2+\varphi(n)\leq 4
\end{equation*}
\noindent where $\varphi(n)\ra 0$ as $n\ra \infty$.
\end{pro}
Let $t_n$ be the smallest integer larger or equal to $(n/2)\log n$.
Then  the result above and an easy lower bound discussed below imply that
$$1\le \lim_{n\ra\infty}d_2(q^{(t_n)},u)\le 2.$$
It is quite rare to be able to capture the mixing time of a chain with such  precision.

\begin{proof}
Let $\mathcal{C}\subset S_n$ be the conjugacy class of transpositions,
$\tau\in\mathcal C$ be a transposition,
and 
$$r(\rho_{\lambda})=\frac{\chi_{\rho_{\lambda}}(\tau)}
{\mbox{d}_{\rho_{\lambda}}}.$$
Proposition \ref{dist-character} gives
\begin{equation}\label{eqn-dist-rt}
d_2(q^{(t)},u)^2
=\sum_{\rho_{\lambda}\neq 1}\mbox{d}_{\rho_{\lambda}}^2\left(\frac{1}{n}+\frac{n-1}{n}r(\rho_{\lambda})\right)^{2t}
\end{equation}
\noindent In \cite{DS} it is shown that
\begin{eqnarray}\label{r}
r(\rho_{\lambda})\leq\left\{\begin{array}{ll}
1-\frac{2(n-\lambda_1)(\lambda_1+1)}{n(n-1)}& \textrm{if $\lambda_1\geq n/2$}\\
\frac{\lambda_1-1}{n-1}& \textrm{always.}
\end{array}\right.
\end{eqnarray}
It follows from equations (\ref{dim}), (\ref{eqn-dist-rt}), and (\ref{r}) that
\begin{eqnarray*}
d_2(q^{(t)},u)^2&=&
\sum_{j=1}^{n-1}\sum_{\rho_{\lambda}\atop \lambda_1=n-j}d_{\rho_{\lambda}}^2\left(\frac{1}{n}+\frac{n-1}{n}r(\rho_{\lambda})\right)^{2t}\\
&\leq&\sum_{j=1}^{\frac{n}{2}}\left(\frac{n!}{(n-j)!}\right)^2\frac{1}{j!}\left(1-\frac{2j}{n}\left(1-\frac{j-1}{n}\right)\right)^{2t}\\
&&+\sum_{{j=\frac{n}{2}}}^{n-1}\left(\frac{n!}{(n-j)!}\right)^2\frac{1}{j!}\left(1-\frac{j}{n}\right)^{2t}.
\end{eqnarray*}

\noindent Note that for $1\leq j\leq \frac{n}{2}$ we have that
$1-\frac{2j}{n}\left(1-\frac{j-1}{n}\right)\leq 1-\frac{2}{n}\leq e^{-2/n}.$
So for $t\geq (n/2)(\log{n}+c)$
\begin{equation*}
d_2(q^{(t)},u)^2\leq e^{-2c}\left(\sum_{j=1}^{n/2}A_j+\sum_{j=n/2}^{n-1}B_j\right),
\end{equation*}
\noindent where

\begin{eqnarray}
A_j&=&
\left(\frac{n!}{(n-j)!}\right)^2\frac{1}{j!}\left(1-\frac{2j}{n}\left(1-\frac{j-1}{n}\right)\right)^{n\log{n}}\label{Aj}\\
B_j&=&\left(\frac{n!}{(n-j)!}\right)^2\frac{1}{j!}\left(1-\frac{j}{n}\right)^{n\log{n}}.\label{Bj}
\end{eqnarray}

\noindent Consider the following two technical propositions.
\begin{pro}\label{lem-tech1-rt}
Set  
$\varphi_0(n)=\sum_{j=1}^{\lfloor n/4\rfloor}A_j
\;\;\text{and}\;\; 
\varphi_1(n)=\sum_{j=\lceil n/4\rceil}^{\lfloor n/2\rfloor} A_j$.
For $n\geq 14$
\begin{equation*}
\varphi_0(n)\leq 2 
\;\;\text{and}\;\;
\varphi_1(n)\leq \exp\left\{2-\frac{1}{6}n\log{n}\right\}.
\end{equation*}
\end{pro}

\begin{pro}\label{lem-tech2-rt}
Set $\varphi_2(n)=\sum_{j=\lceil n/2\rceil}^n B_j$. For $n\geq 9$
\begin{equation*}
\varphi_2(n)\leq \exp\left\{1-\frac{3}{1000}n\log{n}\right\}.
\end{equation*}
\end{pro}

\noindent Propositions \ref{lem-tech1-rt} and \ref{lem-tech2-rt} give that for $n\geq 14$
\begin{equation*}
d_2(q^{t},u)^2
=e^{-2c}(\varphi_0+\varphi_1+\varphi_2)\leq e^{-2c}(2+\varphi_1(14)+\varphi_2(14))\leq 4e^{-2c}.
\end{equation*}
\noindent  It also follows that $\varphi_1\ra 0$ and $\varphi_2\ra 0$ as $n\ra \infty$.
\end{proof}

\noindent Next we will show the proofs of the propositions above.

\begin{proof}[Proof of Proposition \ref{lem-tech1-rt}.]
Let $A_j$ be as in equation (\ref{Aj}), the ratio between two
consecutive terms is
\begin{equation*}
\frac{A_{j+1}}{A_j}=\exp\left\{f_n(j)+g_n(j)\right\}
\end{equation*}
\noindent where
\begin{eqnarray*}
f_n(j)&=&2\log(n-j)-\log(j+1) \\
g_n(j)&=&n\log{n}\log\left(\frac{n^2-2(j+1)n+2j(j+1)}{n^2-2jn+2j(j-1)}\right).
\end{eqnarray*}

\noindent Taking derivatives gives
\begin{eqnarray*}
f_n'(j)&=&-\frac{2}{n-j}-\frac{1}{j+1}\\
g_n'(j)&=&\frac{4(n\log{n})(2jn-2j^2-n)}{(n^2-2jn+2j^2-2j)(n^2-2jn-2n+2j^2+2j)}.
\end{eqnarray*}

\noindent Note that for $1\leq j\leq n/4$ and $n\geq 4$ we
have that $f_n''(j)=\frac{1}{(j+1)^2}-\frac{2}{(n-j)^2}\geq 0$.  Furthermore,
$g_n''(j)\geq 0$ for $1\leq j\leq n/2$.
The last inequality holds since
for $1\leq j\leq n/2$ the numerator of $g_n'$ is a
positive increasing function of $j$ and the denominator
is a positive decreasing function of $j$.

Set $h_n=f_n+g_n$.  For $1\leq j\leq n/4$ the function $h_n$ is continuous and
has positive second derivative.  It follows that $h_n$ is convex
for said values of $j$, which implies that
\begin{equation*}
h_n(j)\leq\max\left\{h_n(1),h_n\left(n/4\right)\right\}.
\end{equation*}

\noindent Consider the following estimates.
\begin{eqnarray*}
h_n(1)&=&2\log(n-1)-\log{2}+n\left(\log{n}\right)\log\left(1-\frac{2n-4}{n(n-2)}\right)\\
&\leq& 2\log(n-1)-\log{2}-n\left(\log{n}\right)\left(\frac{2n-4}{n(n-2)}\right)\\
&\leq& 2\log(n-1) -\log{2}-2\log{n}
\end{eqnarray*}
\begin{eqnarray*}
h_n(n/4)&=&2\log\left(\frac{3n}{4}\right)-\log\left(\frac{n+4}{4}\right)
+n\left(\log{n}\right)\log\left(1-\frac{8}{5n-4}\right)\\
&\leq&2\log\left(\frac{3n}{4}\right)-\log\left(\frac{n+4}{4}\right)-n\left(\log{n}\right)\left(\frac{8}{5n-4}\right)\\
&\leq&2\log\left(\frac{3n}{4}\right)-\log\left(\frac{n+4}{4}\right)-\frac{8\log{n}}{5}\\
&\leq&2\log{3}-\log{4}+\frac{2}{5}\log{n}-\log(n+4)
\end{eqnarray*}

\noindent For $n\geq 2$\; $h_n(1)$ and $h_n(n/4)$ are decreasing functions of $n$
less than $-\log{2}$.  Since  $A_1=n^2\left(1-2/n\right)^{n\log{n}}\leq 1$,
it follows that for $1\leq j\leq n/4$
\begin{equation*}
A_j\leq (1/2)^{j-1} A_1\leq (1/2)^{j-1}.
\end{equation*}
\noindent We can now state the first part of Proposition \ref{lem-tech1-rt}
\begin{equation*}
\varphi_0(n)
=\sum_{j=0}^{\frac{n}{4}}A_j\leq\sum_{j=1}^{\infty}\left(\frac{1}{2}\right)^{j}=2.
\end{equation*}

Next we bound $A_j$ for $n/4 \leq j\leq n/2$.
It is not hard to show that $f_n'''(j)\leq 0$, so
for the values of $j$ above
\begin{equation*}
f_n'(j)\geq\min\{f_n'(n/4),f_n'(n/2)\}.
\end{equation*}
\noindent Note that
\begin{equation*}
f_n'(n/4)=-\frac{4(5n+8)}{3n(n+4)}
\;\;\text{and}\;\;
f_n'(n/2)=-\frac{2(3n+4)}{n(n+2)}.
\end{equation*}
For $n\geq 14$, $f_n'(n/2)\geq f_n'(n/4)$.
Recall that for $1\leq j\leq n/2$ we had that $g_n''\geq 0$.
It follows that for $n/4\leq j\leq n/2$,
\begin{equation*}
h_n'(j)=f_n'(j)+g_n'(j)\geq f_n'(n/4)+g_n'(n/4)\geq 0.
\end{equation*}

Above we showed that $h_n(n/4)\leq 0$.  For $n\geq 3$ we have that
$h_n(n/2)= 2\log\left(\frac{n}{2}\right)-\log\left(\frac{n}{2}+1\right)\geq 0$,
so there must be a unique point $x\in[n/4,n/2]$ such that
$h_n(x)=0$. If $n/4\leq j\leq x$ then
$(A_{j+1}/A_j)\leq 1$.  If $x\leq j\leq n/2$ then
$(A_{j+1}/A_j)\geq 1$.  So for $n/4\leq j\leq n/2$
\begin{equation*}
A_j\leq\max\left\{A_{\frac{n}{4}},A_{\frac{n}{2}}\right\}.
\end{equation*}

\noindent In \cite{Fe} a proof of Stirling's formula shows that
\begin{eqnarray}
\sqrt{2\pi n}\left(\frac{n}{e}\right)^n\leq n!\leq e^{\frac{1}{12n}}\sqrt{2\pi n}\left(\frac{n}{e}\right)^n.\label{Sterling}
\end{eqnarray}

\noindent To determine the largest value among $A_{\frac{n}{4}}$ and $A_{\frac{n}{2}}$ 
we consider the ratio
\begin{eqnarray*}
\frac{A_{\frac{n}{4}}}{A_{\frac{n}{2}}}
&=&\left(\frac{\left(\frac{n}{2}\right)!}{\left(\frac{3n}{4}\right)!}\right)^{2}
\left(\frac{\left(\frac{n}{2}\right)!}{\left(\frac{n}{4}\right)!}\right)
\left(\frac{5n-4}{4n-8}\right)^{n\log{n}}\\
&\leq& \left(\frac{2\sqrt{2}e^{\frac{1}{4n}}}{3}\right)\left(\frac{e^{\frac{n}{2}}4^n}{n^{\frac{n}{2}}3^{\frac{3n}{2}}}\right)
\left(\frac{n^{\frac{n}{4}}}{e^{\frac{n}{4}}}\right)\nonumber\left(\frac{5n-4}{4n-8}\right)^{n\log{n}}\\
&=&\left(\frac{2\sqrt{2}e^{\frac{1}{4n}}}{3}\right)\left(\frac{4}{3^{\frac{3}{2}}}\right)^n
\left(\frac{e^{\frac{1}{4}}}{n^{\frac{1}{4}}}\right)^n n^{n\log\left(\frac{5n-4}{4n-8}\right)}
=\left(\frac{2\sqrt{2}e^{\frac{1}{4n}}}{3}\right)\exp\left\{l(n)\right\}
\end{eqnarray*}

\noindent where
$l(n)=n\left(\log\left(\frac{4}{3^{3/2}}\right)+\frac{1}{4}\right)+n\log{n}\left(\log\left(\frac{5n-4}{4n-8}\right)-\frac{1}{4}\right)$.
For $n\geq 47$ we have that $\left(\log\left(\frac{5n-4}{4n-8}\right)-\frac{1}{4}\right)\leq 0$
which implies that $l(n)\leq 0$.  If $5\leq n\leq 47$ one can check
that $l(n)\leq 1$.  So for $n\geq 5$ we have that
$(A_{\frac{n}{4}}/A_{\frac{n}{2}})\leq e$ which in turn implies that
$\sum_{j=n/4}^{n/2}A_j\leq (e/4)nA_{\frac{n}{2}}$.
By using Stirling's formula to estimate $A_{\frac{n}{2}}$  we get
\begin{eqnarray*}
\left(\frac{en}{4}\right)A_{\frac{n}{2}}&=&\left(\frac{en}{4}\right)\left(\frac{n!}{\left(\frac{n}{2}\right)!}\right)^2
\left(\frac{1}{\left(\frac{n}{2}\right)!}\right)\left(\frac{n-2}{2n}\right)^{n\log{n}}\\
&=&\left(\frac{en}{4}\right)\left(\frac{e^{\frac{1}{12n}}\sqrt{2}n^{\frac{n}{2}}2^{\frac{n}{2}}}{e^{\frac{n}{2}}}\right)^2
\left(\frac{2^{\frac{n}{2}}e^{\frac{n}{2}}}{n^{\frac{n}{2}}\sqrt{\pi n}}\right)\left(\frac{n-2}{2n}\right)^{n\log{n}}\\
&=&\left(\frac{e^{1+\frac{1}{6n}}\sqrt{n}}{2\sqrt{\pi}}\right)\left(\frac{n^{\frac{n}{2}}2^{\frac{3n}{2}}}{e^{\frac{n}{2}}}\right)
\left(\frac{n-2}{2n}\right)^{n\log{n}}\\
&=&\left(\frac{e^{1+\frac{6}{n}}}{2\sqrt{\pi}}\right)\exp\left\{nf(n)\log{n}\right\}
\end{eqnarray*}

\noindent where
$f(n)=\left(\frac{3\log{2}-1}{2}\right)\left(\log{n}\right)^{-1}+\frac{1}{2n}+\frac{1}{2}+\log\left(\frac{n-2}{2n}\right)$.
Computing the derivative gives us that

\begin{equation*}
f'(n)=\frac{-(3n\log{2}-1)n^2+2(3\log{2}-1)n+3n(\log{n})^2+2(\log{n})^2}{2n^2(\log{n})^2(n-2)}.
\end{equation*}

\noindent Note that $f'\geq 0$ for $n>2$, so $f(n)\leq \lim_{n\rightarrow\infty} f(n) = \frac{1}{2}-\log{2}.$
We can now concluded that for $n\geq 5$
\begin{eqnarray*}
\varphi_1(n)&\leq&\left(\frac{en}{4}\right)A_{\frac{n}{2}}
\leq\left(\frac{e^{1+\frac{6}{n}}}{2\sqrt{\pi}}\right)\exp\left\{(n\log{n})\left(\frac{1}{2}-\log{2}\right)\right\}\\
&\leq&\frac{e^3}{2\sqrt{\pi}}\exp\left\{-\frac{1}{6}n\log{n}\right\}
\leq \exp\left\{2-\frac{1}{6}n\log{n}\right\}.
\end{eqnarray*}

\end{proof}

\begin{proof}[Proof of Proposition \ref{lem-tech2-rt}.]

Let $B_j$ be as in equation (\ref{Bj}).
If $n>2$ and $n/2\leq j\leq n$ we can estimate the ratio of
$B_j$ and $B_{j+1}$ by
\begin{equation*}
\frac{B_{j+1}}{B_j}=\frac{(n-j)^2}{(j+1)}\left(1-\frac{1}{n-j}\right)^{n\log{n}}
\leq 2n\left(1-\frac{2}{n}\right)^{n\log{n}}\leq\frac{2}{n}.
\end{equation*}
We get that
$B_j\leq(2/n)^{j-n/2}B_{\frac{n}{2}}$.  It follows that
\begin{equation*}
\varphi_2(n)=
\sum_{j=\frac{n}{2}}^{n} B_j
\leq B_{\frac{n}{2}}\sum_{j=\frac{n}{2}}^{n}(2/n)^{j-\frac{n}{2}}
\leq B_{\frac{n}{2}}\sum_{j=0}^{\infty}(2/n)^{j}
=\frac{B_{\frac{n}{2}}}{1-(2/n)}.
\end{equation*}

\noindent Using Stirling's formula we can bound $B_{\frac{n}{2}}$ to get
\begin{eqnarray*}
\left(\frac{1}{1-2/n}\right)B_{\frac{n}{2}}
&=&\left(\frac{1}{1-2/n}\right)\left(\frac{n!}{\left(\frac{n}{2}\right)!}\right)^2
\left(\frac{1}{\left(\frac{n}{2}\right)!}\right)\left(\frac{1}{2}\right)^{n\log{n}}\\
&\leq&\left(\frac{1}{1-2/n}\right) \left(\frac{2e^{\frac{1}{6n}}}{\sqrt{\pi n}}\right)
\left(\frac{n^{\frac{n}{2}}2^{\frac{3n}{2}}}{e^{\frac{n}{2}}}\right)
\left(\frac{1}{2}\right)^{n\log{n}}\\
\\
&=&\left(\frac{1}{1-2/n}\right)\left(\frac{2e^{\frac{1}{6n}}}{\sqrt{\pi}}\right)
\exp\left\{n(\log{n})b(n)\right\}\\
\end{eqnarray*}

\noindent where
$b(n)=-\frac{1}{2n}+\left(\frac{3\log{2}-1}{2}\right)\left(\log{n}\right)^{-1}
+\frac{1}{2}+\log\left(\frac{1}{2}\right)$.
Taking derivatives gives that
\begin{equation*}
b'(n)=\frac{(\log{n})^2-n(3\log{2}-1)}{2n^2(\log{n})^2}
\;\;\text{so}\;\;
b'(n)\leq 0 \;\;\text{for $n\geq1$.}
\end{equation*}
\noindent For $n\geq 9$ we have that $b(n)\leq b(9)< -\frac{3}{1000}$.
Furthermore, for $n\geq 9$ the function 
\begin{equation*}
g(n)=\left(\frac{1}{1-2/n}\right)\left(\frac{2e^{\frac{1}{6n}}}{\sqrt{\pi}}\right)
\exp\left\{-\frac{3}{1000}n\log{n}\right\}
\end{equation*}
is decreasing. So for $n\geq 9$ 
\begin{equation*}
\varphi_2(n)\leq g(n)\leq \exp\left\{1-\frac{3}{1000}n\log{n}\right\}.
\end{equation*}
\end{proof}

A lower bound for the $\chi$-square distance is obtain by writing
$d_2(q^{(k)},u)^2\ge (n-1)^2(1-2/n)^{2k}$ which uses the
term associated to the Young diagram $(n-1,1)$. Alternatively, let $\varphi(\sigma)$ 
be the function with denotes the number of fixed points of $\sigma$.  One can check by 
inspection that $\varphi -1$ is a normalized eigenfunction associated with the eigenvalue $(1-2/n)$.
This gives the same $\ell^2$ lower bound.

Concerning total variation lower bounds, \cite{Dia} shows that for
any $c>0$ and $t\geq(n/2)(\log{n}-c)$
\begin{equation*}
\lim_{n\ra\infty}d_{\mbox{\tiny{TV}}}(q^{(t)},u)
\geq 1/e-e^{-e^{-2c}}
\end{equation*}
A slight modification of  the argument used in
\cite{Dia} (as presented above in the proof of Proposition
\ref{thm-tr-lower}) yields  the following proposition.

\begin{pro}\label{rt-lowerbound}
Let $q$ be the random transposition measure on the group $S_n$.
For any sequence $k_n$ such that $(2k_n-n\log{n})/n$ tends to
$-\infty$ as $n$ tends to $\infty$, we have
\begin{equation*}
\lim_{n\ra \infty}d_2(q^{(k_n)},u)=\infty\;\;\text{and}\;\;
\lim_{n\ra \infty}d_{\mbox{\em \tiny{TV}}}(q^{(k_n)},u)=1.
\end{equation*}
\end{pro}

\subsection{Random transposition in continuous time}
This section is devoted to the continuous time version of random transposition. 
There is no proof in the literature that the continuous time random transposition 
shuffle has a $\ell^2$ cutoff at time $(n/2) \log n$.
One reason is that the fact that it does not automatically follow
from the discrete time result is often overlooked.
In fact, getting an upper bound in the continuous time case turns out to be somewhat 
more difficult than in the discrete case.   The difficulty comes from handling 
the contribution of the small eigenvalues of $q$.  Compare with what is
proved below for conjugacy classes with less fixed points, e.g., $4$-cycles. One 
very good reason to want to have a good $\ell^2$ upper-bound in continuous time for 
random transposition is that it yields better result when used with the comparison 
technique of \cite{DSC} to study other chains. See Section \ref{sec-RI} below.

\begin{pro}\label{pro-cont-tr}
Let $h_t$ be the law of the continuous time process associated to the
random transposition measure $q$. If $n\geq 10$, $c\geq 2$ then
for $t\geq (n/2)(\log{n}+c)$
\begin{equation*}
2d_{\mbox{\em \tiny TV}}(h_t,u)\le d_2(h_t,u)\leq e^{-(c-2)}.
\end{equation*}
Moreover, if $t_n$ is any sequence of time such that
$(2t_n-n\log n)/n$ tends to $-\infty$ as $n$ tends to infinity, we have
$$\lim_{n\ra \infty}d_{\mbox{\em \tiny TV}}(h_{t_n},u)=1,\;\;\;
 \lim_{n\ra \infty}d_2(h_{t_n},u)=\infty.$$
\end{pro}
Let us observe that we are not able to show that $d_2(h_{(n/2)\log n},u)$
is bounded above independently of $n$ (compare with the discrete time case).

\begin{proof} The lower bound in $\ell^2$ follows from the same argument used in the 
discrete time case. The lower bound in total variation is known.  See, e.g.,
in \cite{SCMZ}. We focus on the upper bound in $\ell^2$.

Let $\mathcal{C}\subset S_n$ be the conjugacy class of transpositions.
Proposition \ref{dist-character} implies that
\begin{equation}\label{eqn-cont-rt}
d_2(h_t,u)^2=\sum_{\rho_{\lambda}\neq 1}\mbox{d}_{\rho_{\lambda}}^2
\exp\left\{-2t\left(1-\frac{1}{n}-\frac{n-1}{n}r(\rho_{\lambda})\right)\right\}
\end{equation}
\noindent where
$r(\rho_{\lambda})=\chi_{\rho_{\lambda}}(\tau)/\mbox{d}_{\rho_{\lambda}}$
and $\tau$ is a transposition.
Using equations (\ref{dim}), (\ref{eqn-cont-rt}), and (\ref{r}) we get that
for $t\geq (n/2)(\log{n}+c)$
\begin{eqnarray*}
d_2(h_t,u)^2&=&\sum_{j=1}^{n-1}\sum_{\rho_{\lambda}\neq 1\atop \lambda_1=n-j}
\mbox{d}_{\rho_{\lambda}}^2
\exp\left\{-2t\left(1-\frac{1}{n}-\frac{n-1}{n}r(\rho_{\lambda})\right)\right\}\\
&\leq&\sum_{j=1}^{n/2}\left(\frac{n!}{(n-j)!}\right)^2\frac{1}{j!}
\exp\left\{-2t\left(\frac{2j}{n}\right)\left(1-\frac{j-1}{n}\right)\right\}\\
&&+\sum_{j=n/2}^{n-1}\left(\frac{n!}{(n-j)!}\right)^2\frac{1}{j!}
\exp\left\{-2t\left(\frac{j}{n}\right)\right\}\\
&\leq& \sum_{j=1}^{n/2}\left(\frac{n!}{(n-j)!}\right)^2\frac{1}{j!}
\exp\left\{-2j(\log{n}+c)\left(1-\frac{j-1}{n}\right)\right\}\\
&&+\sum_{j=n/2}^{n-1}\left(\frac{n!}{(n-j)!}\right)^2\frac{1}{j!}
\exp\left\{-j(\log{n}+c)\right\}
\end{eqnarray*}

\noindent Note that for $c\geq 2$ and $j\leq n/2$ we have 
$-2cj\left(1-\frac{j-1}{n}\right)\leq -2c-2j+4.$
It follows that for $t\geq (n/2)(\log{n}+c)$
\begin{equation*}
d_2(h_t,u)^2\leq e^{-2(c-2)}\left(\sum_{j=1}^{n/2}A_j+\sum_{j=n/2}^nB_j\right)
\end{equation*}

\noindent where
\begin{eqnarray}
A_j&=&\left(\frac{n!}{(n-j)!}\right)^2\frac{1}{j!}
\exp\left\{-2j\log{n}\left(1-\frac{j}{n}\right)-2j\right\}\label{Aj-cont}\\
B_j&=&\left(\frac{n!}{(n-j)!}\right)\frac{1}{j!}
\exp\{-j\log{n}-2j\}.\label{Bj-cont}
\end{eqnarray}

\noindent Consider the following technical lemmas.

\begin{lem}\label{lem-Aj-cont}
For $n\geq 10$ then
$\sum_{j=1}^{n/4} A_j\leq 2/3$ and $\sum_{j=n/4}^{n/2}A_j \leq 1/4$.
\end{lem}

\begin{lem}\label{lem-Bj-cont}
Set $\gamma(n)=\sum_{j=n/2}^{n}B_j$. 
For $n\geq 2$ 
\begin{equation*}
\gamma(n)\leq 2\left(\frac{2}{e}\right)^{\frac{3n}{2}}. 
\end{equation*}
\end{lem}

\noindent It follows from the lemmas above that for $n\geq 10$
\begin{eqnarray*}
d_2(h_t,u)^2
&\leq& e^{-2(c-2)}\left(\sum_{j=1}^{n/4}A_j+\sum_{j=n/4}^{n/2}A_j+\gamma(10)\right)\\
&\leq& e^{-2(c-2)}\left(2/3+1/4+2(2/e)^{15}\right)\\
&\leq&  e^{-2(c-2)}.
\end{eqnarray*}
\end{proof}

\begin{proof}[Proof of Lemma \ref{lem-Aj-cont}]
Let $A_j$ be as in equation (\ref{Aj-cont}).  For $1\leq j<n/2$ the
ratio of two consecutive terms is given by
\begin{equation*}
\frac{A_{j+1}}{A_j}
=\frac{(n-j)^2}{(j+1)}
 \exp\left\{-\left(\frac{2\log{n}}{n}\right)(n-2j-1)-2\right\}
=\exp\{f_n(j)\}
\end{equation*}
\noindent where
\begin{equation}\label{ratio-A-cont}
f_n(j)
=2\log(n-j)-\log(j+1)-\left(\frac{2\log{n}}{n}\right)(n-2j-1)-2.
\end{equation}

\noindent Taking derivatives gives
\begin{eqnarray*}
f_n'(j)&=&-\frac{2}{n-j}-\frac{1}{j+1}+\frac{4\log{n}}{n}\\
f_n''(j)&=&-\frac{2}{(n-j)^2}+\frac{1}{(j+1)^2}.
\end{eqnarray*}

\noindent  Let $n\geq 4$ and $1\leq x\leq n/4$. For these
values of $n$ and $x$ we get that $f_n$ is convex since
$f_n''$ is a decreasing function and $f_n''(x)\geq f_n''(n/4)\geq 0$.

\begin{equation*}
\frac{A_{x+1}}{A_x}
=\exp\left(\max \left\{f_n(1),f_n\left(\frac{n}{4}\right)\right\}\right).
\end{equation*}

\noindent If $n\geq 2$ we have the estimates
\begin{eqnarray*}
f_n(1)&=&2\log(n-1)-\log{2}-\left(\frac{2\log{n}}{n}\right)(n-3)-2
\leq -\log{2}\\
f_n(n/4)
&=&2\log\left(\frac{3n}{4}\right)-\log\left(\frac{n}{4}+1\right)
    -\left(\frac{2\log{n}}{n}\right)\left(\frac{n}{2}-1\right)-2
\leq 2\log\left(\frac{3}{4}\right).
\end{eqnarray*}

\noindent Since $-\log{2}\leq 2\log(3/4)$, we get that
$A_x\leq(9/16)^{x-1}$.  It now follows that
\begin{equation*}
\sum_{j=1}^{\frac{n}{4}}A_j
\leq A_1\sum_{j=0}^{\infty}\left(\frac{9}{16}\right)^j
=\left(\frac{16}{7}\right)A_1.
\end{equation*}

\noindent For $n\geq 4$ we get 
$A_1=n^2\exp\left\{-\frac{2(n-1)}{n}\log{n}-2\right\}\leq 2e^{-2}$.
This gives that 
\begin{equation*}
\sum_{j=1}^{\frac{n}{4}}A_j\leq (32/7)e^{-2}\leq 2/3.
\end{equation*}

For the next part of the proof let $n\geq 10$.  Recall that
$f_n^{''}$ is a decreasing function, which implies that
for $n/4\leq j\leq n/2$ 
\begin{equation*}
f_n'(j)\geq\min\left\{f_n'(n/4),f_n'(n/2)\right\}\geq 0
\end{equation*}
\noindent where the last inequality holds since $n\geq 10$.
Since $f_n$ is an increasing function with $f_n(n/4)\leq 0$
and $f_n(n/2)\geq 0$ then there exists a
unique point $z\in [n/4,n/2]$ such that $f_n(z)=0$. It follows
that if $n/4\leq j\leq z$ then $A_j\leq A_{n/4}$ and if
$z\leq j\leq n/2$ then $A_j\leq A_{n/2}$.  Combining these
two inequalities gives us that for  $n/4\leq j\leq n/2$ 
\begin{equation*}
A_j\leq\max\{A_{n/4},A_{n/2}\}.
\end{equation*}

\noindent To compare $A_{n/4}$ and $A_{n/2}$ we use
Stirling's formula (\ref{Sterling}).  For $n\geq 2$
\begin{eqnarray*}
A_{n/4}
&=&\left(\frac{n!}{\left(\frac{3n}{4}\right)!}\right)^2
   \left(\frac{1}{\left(\frac{n}{4}\right)!}\right)
   \exp\left\{-\left(\frac{3n}{8}\right)\log{n}-\frac{n}{2}\right\}\\
&\leq&\left(\frac{e^{\frac{1}{6n}}4\sqrt{2}}{3\sqrt{\pi n}}\right)
   n^{-\frac{n}{8}}\left(\frac{4}{3}\right)^{\frac{3n}{2}}
   4^{\frac{n}{4}}e^{-\frac{3n}{4}}
\leq n^{-\frac{n}{8}}\left(\frac{4}{3}\right)^{\frac{3n}{2}}
     4^{\frac{n}{4}}e^{-\frac{3n}{4}}\\
A_{n/2}
&=&\left(\frac{n!}{\left(\frac{n}{2}\right)!}\right)^2
   \left(\frac{1}{\left(\frac{n}{2}\right)!}\right)
   \exp\left\{-\left(\frac{n}{2}\right)\log{n}-n\right\}\\
&\leq&\left(\frac{2e^{\frac{1}{6n}}}{\sqrt{\pi n}}\right) 2^{\frac{3n}{2}}
      e^{-\frac{3n}{2}}\leq 2^{\frac{3n}{2}}e^{-\frac{3n}{2}}.
\end{eqnarray*}

\noindent It follows that
\begin{equation*}
nA_{n/4}\leq \exp\{\phi_1(n)\}\;\;\text{and}\;\;
nA_{n/2}=\exp\left\{\phi_2(n)\right\}, 
\end{equation*}
where
\begin{eqnarray*}
\phi_1(n)&=&\log{n}-\left(\frac{n}{8}\right)\log{n}-\left(\frac{3n}{4}\right)
+\left(\frac{3n}{2}\right)\log\left(\frac{4}{3}\right)
+\left(\frac{n}{4}\right)\log{4}\\
\phi_2(n)&=&\log{n}-\left(\frac{3n}{2}\right)+\left(\frac{3n}{2}\right)\log{2}.
\end{eqnarray*}

\noindent For $n\geq 10$ we have  $\phi_1(n)\leq 0$
and $\phi_2(n)\leq 0$ which implies that
\begin{equation*}
\sum_{j=n/4}^{n/2}A_j
\leq\left(\frac{1}{4}\right)
\max\left\{nA_{\frac{n}{4}},nA_{\frac{n}{2}}\right\}\leq \frac{1}{4}.
\end{equation*}
\end{proof}

\begin{proof}[Proof of Lemma \ref{lem-Bj-cont}]
Let $n/2\leq j\leq n$ and $B_j$ be as in equation (\ref{Bj-cont}).
As usual, consider we consider the ratio between two consecutive
term
\begin{equation*}
\frac{B_{j+1}}{B_j}=\frac{(n-j)^2}{(j+1)}\exp\left\{-\log{n}-2\right\}
\leq\left(\frac{n}{2}\right)\exp\{-\log{n}-2\}\leq\frac{1}{2}.
\end{equation*}

\noindent Note that
$B_j\leq \left(\frac{1}{2}\right)^{n/2-j}B_{n/2}$,
which implies that
\begin{equation*}
\gamma(n)=\sum_{j=n/2}^nB_j\leq B_{n/2}\sum_{j=0}^{\infty}\left(\frac{1}{2}\right)^j
=2B_{n/2}.
\end{equation*}

\noindent Since
$B_{\frac{n}{2}}=A_{\frac{n}{2}}\leq (2/e)^{\frac{3n}{2}}$ then for
$n\geq 2$ we have that $\gamma(n)\leq 2(2/e)^{\frac{3n}{2}}.$
\end{proof}

\subsection{Random Insertions}\label{sec-RI}
In the random insertion shuffle for a deck of $n$ cards,
one picks out a random card and inserts it back into the deck at a
random position.  This shuffle is modeled by the measure $q$ on
the $S_n$ given by
\begin{eqnarray}\label{insertion-measure}
q(\tau)=\left\{\begin{array}{ll}
1/n & \textrm{if $\tau=e$}\\
2/n^2 & \textrm{if $\tau=c_{i,j}$ s.t. $1\leq i,j\leq n$ and $|i-j|=1$}\\
1/n^2 & \textrm{if $\tau=c_{i,j}$ s.t. $1\leq i,j\leq n$ and $|i-j|>1$}\\
0 & \textrm{otherwise.}
\end{array}\right.
\end{eqnarray}

\noindent where $c_{i,j}$ denotes the cycle created by
taking the card in position $i$ and inserting it into position $j$.
A formal definition is given by

\begin{equation}
c_{i,j}=\left\{\begin{array}{ll}
e & \textrm{if $i=j$}\nonumber\\
(j,j-1,\dots,i+1,i) &\textrm{if $1\leq i<j\leq n$}\label{eq-cycles}\\
(j,j+1,\dots,i-1,i) &\textrm{if $1\leq j<i\leq n.$}\nonumber
\end{array}\right.
\end{equation}

Random insertion is the first of the shuffles discussed in this paper
for which  it is not know whether there is a total variation cutoff or not
although it is strongly believed that there is one. The results of
\cite{GY,GYSC} show that there is a cutoff in $\ell^2$ but the exact 
cutoff time is not known. What is known and follows from \cite{DSC} is that
there is a pre-cutoff (in both total variation and $\ell^2$) at time
$n\log n$. Finding the precise $\ell^2$ cutoff time and proving a cutoff in total 
variation are challenging open problem that have been investigated (but not solved) 
in \cite{UR} by Uyemura-Reyes.

\begin{theo}{\bf (Diaconis and Saloff-Coste \cite{DSC} and Uyemura-Reyes \cite{UR})}\label{thm-UR}
Let $q$ be the random insertion measure on $S_n$ defined above.
For $c>0$ and $t\geq 4n(\log{n}+c)$ there exists a constant $B$
such that
\begin{equation*}
d_2(q^{(t)},u)\leq Be^{-c}.
\end{equation*}
\noindent For any sequence $(k_n)$ such that $(2k_n-n\log{n})/n$
tends to $-\infty$ as $n$ tends to $\infty$ then
\begin{equation*}
d_{\mbox{\em \tiny{TV}}}(q^{(k_n)},u)\ra 1 \;\;\text{and}\;\;
d_2(q^{(k_n)},u)\ra\infty.
\end{equation*}
\end{theo}

In \cite{DSC} the mixing time in Theorem \ref{thm-UR} is shown to be
$\mathcal{O}(n\log{n})$ while in \cite{UR} the more precise upper bound 
given in Theorem \ref{thm-UR} is shown.  The proof of the upper bound in 
Theorem \ref{thm-UR} relies on the comparison techniques developed in \cite{DS-Comparison}.

\begin{defin}
Let $V$ be a state space equipped with a Markov kernel $K$ with reversible  measure $\nu$.
The Dirichlet form associated to $(K,\nu)$ is
\begin{eqnarray*}
\mathcal{E}_{K,\nu}(f,g)&=&\langle (I-K)f,g\rangle_{\nu}=\sum_{x\in V} [(I-K)f(x)]g(x)\nu(x)\\
&=&\frac{1}{2}\sum_{x,y\in V}(f(x)-f(y))(g(x)-g(y))\nu(x)K(x,y)
\end{eqnarray*}
\noindent where $f,g\in \ell^2(\nu,V)$.  In the case where $V$ is a finite group and
$p(x^{-1}y)=K(x,y)$ we set $\mathcal{E}_{p,\nu}=\mathcal{E}_{K,\nu}$.
\end{defin}
\noindent Diaconis and Saloff-Coste show the following theorem.

\begin{theo}{\bf Diaconis and Saloff-Coste, \cite{DS-Comparison}}\label{thm-DS-comp}
Let $q$ and $\tilde{q}$ be the probability measures on a finite group $G$.
Set $\mathcal{E}=\mathcal{E}_{q,u},\tilde{\mathcal{E}}=\mathcal{E}_{\tilde{q},u}$ and
$\beta_i,\tilde{\beta}_i$, $0\leq i\leq |G|-1$  to be the associated
Dirichlet forms and eigenvalues of $q$ and $\tilde{q}$ respectively.
Let $\tilde{h}_t$ to be the law at time $t$ of the continuous time
process associated with $\tilde{q}$.  If there exists a constant $A$ such that
$\tilde{\mathcal{E}}\leq A \mathcal{E}$ then
\begin{equation*}
d_2(q^{(t)},u)^2\leq \beta_{-}^{2t_1}(1+d_2(\tilde{h}_{t_2/A},u)^2)+d_2(\tilde{h}_{t/A},u)^2
\end{equation*}
\noindent where $t=t_1+t_2+1$ and $\beta_{-}=\max\{0,-\beta_{|G|-1}\}$.
\end{theo}

Let $q$ and $\tilde{q}$ be the measures for the random insertion shuffle and
the random transposition shuffle respectively.  In his thesis, Uyemura-Reyes shows that
$A=4$ is the smallest constant such that $\tilde{\mathcal{E}}\leq A \mathcal{E}$.
By noting that $\beta_{-}=0$ we get
\begin{equation}\label{eqn-comp-ri-rt}
d_2(q^{(t)},u)^2\leq d_2(\tilde{h}_{t/4},u)^2.
\end{equation}

\noindent Equation (\ref{eqn-comp-ri-rt}) gives the following corollary to
Proposition \ref{pro-cont-tr}.

\begin{cor}
Let $q$ be the random insertion measure on $S_n$ defined above.
If $n\geq 10$, $c\geq 2$ and $t\geq 2n(\log{n}+c)$ then
\begin{equation*}
d_2(q^{(t)},u)^2\leq e^{-(c-2)}.
\end{equation*}
\noindent For any sequence $(k_n)$ such that $(2k_n-n\log{n})/n$
tends to $-\infty$ as $n$ tends to $\infty$ then
\begin{equation*}
\lim_{n\ra\infty}d_{\mbox{\tiny TV}}(q^{(k_n)},u)= 1.
\end{equation*}
\end{cor}

\begin{proof}
The upper bound results as a corollary to Proposition \ref{pro-cont-tr}
after applying  equation (\ref{eqn-comp-ri-rt}).
The improvement by a factor of $2$ compared to Theorem \ref{thm-UR}
is due to the use of the continuous time random transposition process
in the comparison inequality (\ref{eqn-comp-ri-rt}).

Uyemura-Reyes also proves the total variation lower bound in his thesis
but his proof uses a rather sophisticated argument involving results 
concerning the longest increasing subsequence.
We give an alternative proof of this result based on a technique due to D. Wilson \cite{Wil}.

First note the following result of Uyemura-Reyes. Set $\rho$ to be the
permutation representation.  Let $q$ to be the random insertion measure and $Q$ its
associated Markov kernel such that $Q(x,y)=q(x^{-1}y)$. In \cite{UR} it is shown that
the Fourier transform $\widehat{q}(\rho)$ has an eigenvector $v=(v_0,\dots,v_{n-1})$ where
\begin{equation*}
\widehat{q}(\rho) v=\left(1-\frac{1}{n}\right)v \;\;\; \text{and} \;\;\; v_i=1-\frac{2i}{n-1}.
\end{equation*}

As noticed at the end of Section \ref{sec-RepT}, it follows that
$f_{\rho}(\sigma)=\langle \rho(\sigma)v,v\rangle$, $\sigma\in S_n$, is an
eigenvector of $Q$ with associated eigenvalue $(1-1/n)$.

Computing $f_{\rho}(\sigma)$,  one gets
\begin{eqnarray}
f_{\rho}(\sigma)
&=&\left\langle
\sum_{i=0}^{n-1}\left(1-\frac{2i}{n-1}\right)e_{\sigma(i)},\sum_{j=0}^{n-1}\left(1-\frac{2j}{n-1}\right)e_j
\right\rangle\nonumber\\
&=&\sum_{j=0}^{n-1}\left(1-\frac{2\sigma(j)}{n-1}\right)\left(1-\frac{2j}{n-1}\right)\nonumber\\
&=&\sum_{j=0}^{n-1}1-\frac{2j}{n-1}-\frac{2\sigma(j)}{n-1}+\frac{4\sigma(j) j}{(n-1)^2}\nonumber\\
&=&-n+\frac{4}{(n-1)^2}\sum_{j=0}^{n-1}\sigma(j) j.\label{func-f-rho}
\end{eqnarray}

\noindent Therefore

\begin{equation*}
f^2(\sigma)=n^2-\frac{8n}{(n-1)^2}\sum_{j=0}^{n-1}\sigma(j) j +\frac{16}{(n-1)^4}\sum_{i,j=0}^{n-1}\sigma(i)\sigma(j) i j
\end{equation*}

\noindent and

\begin{eqnarray*}
\sum_{\sigma\in S_n}f^2(\sigma)
&=&n! n^2-\frac{8n}{(n-1)^2}\sum_{j=0}^{n-1} j \sum_{\sigma\in S_n}\sigma(j)
   +\frac{16}{(n-1)^4}\sum_{i,j=0}^{n-1} i j \sum_{\sigma\in S_n}\sigma(i)\sigma(j)\\
&=&n! n^2 -\frac{8n[(n-1)!]}{(n-1)^2} \left(\frac{n(n-1)}{2}\right)^2+\frac{16[(n-2)!]}{(n-1)^4}\left(\frac{n(n-1)}{2}\right)^4\\
&=&n! n^2 -\frac{8n^3(n-1)!}{4}+n^4(n-2)!
=n! \left(\frac{n^2}{n-1}\right)
\end{eqnarray*}

\noindent Next we estimate the supremum norm of the discrete square gradient
of $f_{\rho}$ defined in (\ref{func-f-rho}). The discrete square gradient
of the function $g$ with respect to the kernel $K$ is given by the equation
\begin{equation*}
\left|\nabla g(x)\right|^2=\frac{1}{2}\sum_{y}\left|g(x)-g(y)\right|^2K(x,y).
\end{equation*}

\noindent Calculating the discrete square gradient for $f_{\rho}$ gives us
\begin{eqnarray*}
\left|\nabla f_{\rho}(\sigma)\right|^2
&\leq & \frac{16}{n^2(n-1)^4}\sum_{i,j=0}^{n-1}\left|\sum_{k=0}^{n-1}\sigma^{-1}(k)\left(k-c_{ij}(k)\right)\right|^2
\\
&\leq &
\frac{16}{n^2(n-1)^2}\sum_{i,j=0}^{n-1}\sum_{k=0}^{n-1}\left| k-c_{i,j}(k)\right|^2
\end{eqnarray*}
\noindent where $c_{ij}$ is defined in (\ref{eq-cycles}).
To calculate $k-c_{ij}(k)$ we consider the following two cases.

\noindent{\bf Case 1} If $i<j$
\begin{displaymath}
k-c_{ij}(k)=
\left\{\begin{array}{ll}
i-j & \text{if} \;\;k=i\\
1 & \text{if} \;\;i<k\leq j\\
0 & \text{otherwise.}
\end{array}\right.
\end{displaymath}

\noindent{\bf Case 2} If $j<i$
\begin{displaymath}
k-c_{ij}(k)=
\left\{\begin{array}{ll}
i-j & \text{if}\;\; k=i\\
-1 & \text{if} \;\;j\leq k <i\\
0 & \text{otherwise.}
\end{array}\right.
\end{displaymath}

\noindent It follows that

\begin{eqnarray*}
\left|\nabla f(\sigma)\right|^2
&\leq&\frac{16}{n^2(n-1)^2}\sum_{i<j}\sum_{k=0}^{n-1}\left|k-c_{ij}(k)\right|^2
+\sum_{j<i}\sum_{k=0}^{n-1}\left|k-c_{ij}(k)\right|^2\\
&=&\frac{16}{n^2(n-1)^2}\sum_{i<j}\left((i-j)^2+(j-i)\right)+\sum_{j<i}\left((i-j)^2+(i-j)\right)\\
&=&\frac{16}{n^2(n-1)^2}\sum_{i,j=0}^{n-1}(i-j)^2+|i-j|\\
&\leq &\frac{32}{n^2(n-1)^2}\sum_{i,j=0}^{n-1}(i-j)^2\leq 32.
\end{eqnarray*}

\noindent Lemma $4$ of \cite{Wil} along with the estimate above imply
the stated lower bound in total variation.

\end{proof}

%%%%%%%%%%%%%%%%%%%%%%%%%%%%%%%%%%%%%%%%%%%%%%%%%%%%%%%%%%%%
%%%%%%%%%%%%%%%%%%%%%%%%%%%%%%%%%%%%%%%%%%%%%%%%%%%%%%%%%%%%
\section{Random walks driven by conjugacy classes.}
\subsection{Review of some discrete time results}
In section \ref{sec-RandomTrans} we considered the random walk on $S_n$ driven
by the conjugacy class of transpositions.  More generally,
one can study random walks driven by a fixed conjugacy class.
Recall that $\mathcal{C}$ is a conjugacy class of a group $G$ if
for some $x\in G$ we have that $\mathcal{C}=\{gxg^{-1}:\forall g\in G\}$.

Throughout this section $\mathcal{C}$ will refer to a
conjugacy class in $S_n$ and $\text{supp}(\mathcal{C})$
will denote the support size of $\mathcal{C}$, that is, the number
of points that are not fixed under the action of an element in $\mathcal{C}$.
Conjugacy classes of the symmetric group $S_n$ are described by the cycle structure of
their elements which is often given by a tuple of non-increasing integers
greater than or equal to $2$ and with sum at most $n$. For instance, in
$S_n$ with $n\geq 8$, the tuple $(4,2,2)$ describes the conjugacy class $\mathcal{C}$ 
of those  permutations that are the product of two transpositions
and one $4$-cycle, all  with disjoint supports.  In this example, 
$\text{supp}(\mathcal{C})=8$.

If $\mathcal{C}$ consists of odd permutations, that is, permutations
which can be written as a product of an odd number of transpositions, then
$\mathcal{C}$ generates $S_n$. If $\mathcal{C}$ is even, that is, any element in 
$\mathcal{C}$ can be written as the product of an even number
of transpositions and $\mathcal{C}\neq\{e\}$ then it generates the alternating group $A_n$.
Set $q_{\mathcal{C}}$ to be the measure
\begin{equation}\label{meas-Kcycle}
q_{\mathcal{C}}(\sigma)=\left\{\begin{array}{ll}
\frac{1}{\#\mathcal{C}} & \textrm{if $\sigma\in\mathcal{C}$}\\
0 &\textrm{otherwise}
\end{array}\right.
\end{equation}
where $\# \mathcal{C}$ denotes the number of elements in $\mathcal{C}$.
When $\mathcal{C}$ is an odd conjugacy class the random walk
driven by $q_{\mathcal{C}}$ is be periodic and
$q^t_{\mathcal{C}}$ is supported on $A_n$ when
$t$ is even, and  on $S_n\backslash A_n$ otherwise.
In this case, it is convenient to study the
random walk on $A_n$ driven by $q_{\mathcal{C}}^2$ to avoid periodicity.

The mixing time of these random walks was studied in \cite{LP,MSP,Roi,SP},
among other works. See the discussion in \cite{S-K}.
For simplicity, we describe some of the known results in the case of even conjugacy classes. The same results hold in the odd case, modulo periodicity. In \cite{SP} it is shown that any sequence $(A_n,q_{\mathcal{C}_n})$
has a total variation cutoff at time
\begin{equation*}
t_1(n)=\inf\{k:q^k_{\mathcal{C}_n}(\varphi)\leq \log{n}\}
\end{equation*}
where $\varphi(\sigma)$ is the number of fixed points of $\sigma\in S_n$ and
$q^k_{\mathcal{C}}(\varphi)$ is the expected value of $\varphi$ taken according to
the measure $q^k_{\mathcal{C}}$.  It is well known , see \cite{Dia, Sa}, that
\begin{equation*}
\varphi(\cdot)-1=\chi_{(n-1,1)}(\cdot)=n-1-\text{supp}(\cdot).
\end{equation*}
This implies that $\varphi-1$  is an eigenfunction of $q_{\mathcal{C}}$ with eigenvalue
$\left(\frac{\chi_{(n-1,1)}(\mathcal{C})}{n-1}\right)$.
Thus we can rewrite $t_1(n)$ as
\begin{equation*}
t_1(n)=\inf
\left\{
k:(n-1)\left(1-\frac{\text{supp}(\mathcal{C}_n)}{n-1}\right)^k+1\leq \log{n}
\right\}.
\end{equation*}
When $\text{supp}(\mathcal{C}_n)$ is not too large 
(e.g., $\text{supp}(\mathcal{C}_n)/n=o(1)$)
then $t_1(n)\sim(n/\text{supp}(\mathcal{C}_n))\log{n}$
and when $\text{supp}(\mathcal{C}_n)$ is very large then $t_1(n)$ is $\mathcal{O}(1)$.

Assuming that $\text{supp}(\mathcal{C}_n)\leq n-1$,
\cite{MSP} shows that the random walk driven by
$q_{\mathcal{C}_n}$ has an $\ell^2$ pre-cutoff at time $t_2(n)$ where
\begin{equation*}
\left|t_2(n)-\frac{2\log{n}}{\log(n/(n-\text{supp}(\mathcal{C}_n)+1))}\right|\leq 3.
\end{equation*}
As in the total variation
case, when $\text{supp}(\mathcal{C}_n)$ is not too large then
\begin{equation*}
t_2(n)\sim(n/\text{supp}(\mathcal{C}_n))\log{n}
\end{equation*}
and when
$\text{supp}(\mathcal{C}_n)$ is large we get the at $t_2(n)$ is $\mathcal{O}(1)$.
Here, we will focus on the continuous time process
associated to $q_{\mathcal{C}_n}$.

Corollary 4.1 in \cite{GY} implies that the continuous time
process driven by $q_{\mathcal{C}_n}$  has a total variation mixing time
bounded above by that of the discrete time process. Arguments similar to those
in Chapter 4 of \cite{Ro} give a lower bound for the continuous time process
that is comparable to the upper bound just mentioned.
In particular, for $\text{supp}(\mathcal{C}_n)\leq(n-1)/(\log(n-1)+1)$, these
arguments show that the continuous time chain associated to $q_{\mathcal{C}_n}$
has a total variation cutoff at time $t_1(n)$.

Perhaps surprisingly, we show below that, in $\ell^2$, the mixing time of the
continuous time process has a lower bound of $(n/2)\log{n}$ for any conjugacy
class with $\text{supp}(\mathcal{C}_n)\geq 2$.  A matching upper bound is shown
when $\text{supp}(\mathcal{C}_n)\ra\infty$ as $n\ra\infty$ as well as for the
conjugacy class of $4$-cycles.

\subsection{$\ell^2$ lower bounds in continuous time}\label{sec-lower}

Through out this section $\mathcal C_n$ is a conjugacy class in $S_n$
(or $A_n$)  and $c_n\in \mathcal C_n$ is an arbitrary fixed element in $\mathcal C_n$. 
Recall that $\mbox{supp}(\mathcal C_n)$ is $n-\varphi(c_n)$
where $\varphi(\cdot)$ is the number of fixed points.

\begin{theo} \label{k-cont} For each $n$, set 
$$t_n=\frac{n}{2}\log{n}.$$
For any odd conjugacy class $\mathcal{C}_n\subset S_n$ with $\text{supp}(\mathcal{C}_n)\geq 2$, 
and any $\epsilon\in(0,1)$
\begin{equation*}
\lim_{n\ra\infty} d_2(h_{\mathcal{C}_n,(1-\epsilon)t_n},u_n)=\infty.
\end{equation*}
\end{theo}

In order to understand the mixing time of these continuous time processes
we will again rely on (\ref{main=}) and we will need to estimate
the dimensions and characters of some of the irreducible representations of $S_n$. 
The following  well known definitions and results will help us understand these 
quantities.

\begin{defin}
Let $\lambda$ be a Young diagram with $n$ boxes, as usual, we denote this by 
$\lambda\vdash n$.
The hook at the cell $(i,j)$ is defined as the set of boxes $H_{i,j}$ where
\begin{equation*}
H_{i,j}=\{(i,l):(i,l)\in\lambda,l\geq j\}\cup\{(k,j):(k,j)\in\lambda,k\geq i\}.
\end{equation*}
$H_{i,j}$ has hook length $h_{i,j}=|H_{i,j}|$.
\end{defin}

\begin{theo}{\bf (The Hook formula)}
Let $\lambda$ be a Young diagram with $n$ boxes. Set $d_{\lambda}$ to be
the dimension of the irreducible representation associated to $\lambda$. Then
\begin{eqnarray}\label{hook}
d_{\lambda}=\frac{n!}{\prod_{(i,j)\in\lambda}h_{i,j}}
\end{eqnarray}
\end{theo}

\noindent With the hook formula we can now get an estimate on the dimension of
some representations of $S_n$.

\begin{lem}\label{lem-dim}
Let $n\in\mathbb{N}$ and  $\lambda\vdash n$ be a Young diagram. If $\lambda$ fits into
a rectangle of $s\times t$ boxes, then
\begin{equation*}
d_{\lambda}\geq \left(\frac{n}{e(s+t-1)}\right)^n.
\end{equation*}
\end{lem}

\begin{proof}
Note that any hook in $\lambda$ will be of hook length at most $s+t-1$.
The inequality then follows from Stirling's formula in (\ref{Sterling}) and
hook formula (\ref{hook}).
\end{proof}

We will use the following rather non-trivial bound on character ratios. 
\begin{theo}{\bf (\cite{RS})} \label{theo-RS}
Let $a>0$ be a fixed constant, and let  $\lambda\vdash n$
be a Young diagram with at most $a\sqrt{n}$ rows and columns.  Then there
exists a constant $D=D(a)$ such that
\begin{equation*}
\left|\frac{\chi_{\lambda}(\sigma)}{d_{\rho}}\right|
\leq\left(\frac{D\max\{1,|\sigma|^2/n\}}{\sqrt{n}}\right)^{|\sigma|}
\end{equation*}
\noindent for any  $\sigma\in S_n$ and where $|\sigma|$ is the minimal number of transpositions needed to write $\sigma$ as a product of transpositions.
\end{theo}
Recall that, for any $\sigma\in S_n$ which is not the identity
then $|\sigma|\leq\text{supp}(\sigma)$.

\begin{proof}[Proof of Theorem \ref{k-cont}.]
The idea behind this proof is to write the desired $\ell^2$ distance
as in equation (\ref{main=}) and find an irreducible
representation which has large dimension
and small character.  The representations that are useful in this respect
turn out to be those that have an approximately square shape.

Let $\lambda_n\vdash n$ be a Young diagram that fits into a box of side
$\lceil\sqrt{n}\rceil$, so that $\lambda_n$ looks almost like a square.
By Lemma \ref{lem-dim} and the fact that $\lceil\sqrt{n}\rceil\leq 2\sqrt{n}$
we get that
\begin{equation}\label{eq-square-dim}
d_{\lambda_n}\geq\left(\frac{n}{2e\lceil\sqrt{n}\rceil}\right)^n
\geq\left(\frac{\sqrt{n}}{4e}\right)^n.
\end{equation}

If $\mathcal{C}_n$ is a conjugacy class with $\text{supp}(\mathcal{C}_n)\geq \sqrt{n}$ and $c_n\in \mathcal C_n$,
\cite[Theorem 1]{Roi} yields  a positive constant $q<1$ such that
\begin{equation*}
\left|\frac{\chi_{\lambda_n}(c_n)}{d_{\lambda_n}}\right|
\leq q^{\text{supp}(\mathcal{C}_n)}\leq q^{\sqrt{n}}.
\end{equation*}
If $2\leq\text{supp}(\mathcal{C}_n)\leq\sqrt{n}$, Theorem \ref{theo-RS}  implies that
\begin{equation*}
\left|\frac{\chi_{\lambda_n}(c_n)}{d_{\lambda_n}}\right|\leq\frac{D}{\sqrt{n}}.
\end{equation*}
In either case, we have that
\begin{equation*}
\left|\frac{\chi_{\lambda_n}(c_n)}{d_{\lambda_n}}\right|=o(1).
\end{equation*}

Using (\ref{eq-square-dim}), we obtain that, for any $\epsilon >0$,
\begin{eqnarray*}
\lefteqn{d_{\lambda_n}^2\left\{
-\left(1-\epsilon\right)n\log{n}
\left(1-\frac{\chi_{\lambda_n}(c_n)}{d_{\lambda_n}}\right)
\right\}
\geq}\hspace{1in}  &&  \\
&&\left(\frac{n}{16e^2}\right)^n\exp\left\{
-\left(1-\epsilon\right)n\log{n}\left(1+o(1)\right)
\right\}
\end{eqnarray*}
It now follows from (\ref{main=})
\begin{eqnarray*}
\lim_{n\ra\infty}d_2(h_{\mathcal{C}_n,(1-\epsilon)t_n},u_n)
\geq \lim_{n\ra\infty}d_{\lambda_n}
\exp\left\{-(1-\epsilon)t_n
\left(1-\frac{\chi_{\lambda_n}(c_n)}{d_{\lambda_n}}\right)\right\}
=\infty
\end{eqnarray*}
as desired.
\end{proof}

\noindent Using the same ideas as in the proof of Theorem \ref{k-cont} we
get the following result.

\begin{theo} \label{k-cont-odd}
Let $\mathcal{C}_n$ be a conjugacy class in $A_n$ with $\text{supp}(\mathcal{C}_n)\geq 2$,
and set $\overline{u}_n$ to be the uniform measure on $A_n$.  For any
$\epsilon\in(0,1)$ and $t_n=\frac{n}{2}\log{n}$
\begin{equation*}
\lim_{n\ra\infty} d_2(h_{\mathcal C_n,(1-\epsilon)t_n},\overline{u}_n)=\infty.
\end{equation*}
\end{theo}

\begin{rem}
For $\epsilon>0$, it is interesting to consider the discrete time chain driven by
\begin{equation}\label{q-lazy}
\tilde{q}_{{\mathcal{C}_n},\epsilon}(\sigma)=\left\{
\begin{array}{ll}
\epsilon& \text{if $\sigma$=e}\\
\frac{1-\epsilon}{ \# \mathcal{C}_n}& \text{if $\sigma\in\mathcal{C}_n$}\\
0 & \text{otherwise.}
\end{array}\right .
\end{equation}
When $\epsilon=1/2$, this is often called the lazy chain associated to $q_{\mathcal{C}_n}$.
The arguments used in the proof of Theorem \ref{k-cont} show that the random
walk driven by $\tilde{q}_{\mathcal{C}_n,\epsilon}$ will have a $\ell^2$ mixing time
lower bound of $(n/2)\log_{1/\epsilon}(n)$.

In \cite{S-K} it is conjectured (Conjecture 9.3) that
both the total variation mixing time and the $\ell^2$
mixing time of the random walk driven by
$\tilde{q}_{\mathcal{C}_n}$ will have an upper bound of
$(2n/\text{supp}(\mathcal{C}_n))\log{n}$.
While this is true for in the case of total variation, the results
above show that the bound does not hold for $\ell^2$.
\end{rem}

The proof of Theorem \ref{k-cont} relies on the character
estimates of \cite{RS} and \cite{Roi}.  While these
estimates are very useful, one can construct simple
Young diagram and use the Murnaghan-Nakayama Rule below to get estimates
on the values of  characters at a $k$-cycles, for infinitely many $k$.
This gives a much more accessible proof of a weaker version of 
Theorems \ref{k-cont} and \ref{k-cont-odd}. For further details on the following 
definitions see Section 4.10 in \cite{Sa}.

\begin{defin} A skew hook $\xi$ in a Young diagram is a collection of boxes that
result from the projection of a regular hook along the right boundary of a 
Young diagram.

The leg length of a skew hook $\xi$ is denote by $ll(\xi)$ with
$$ll(\xi)=\text{the number of rows of $\xi$} -1 .$$
\end{defin}

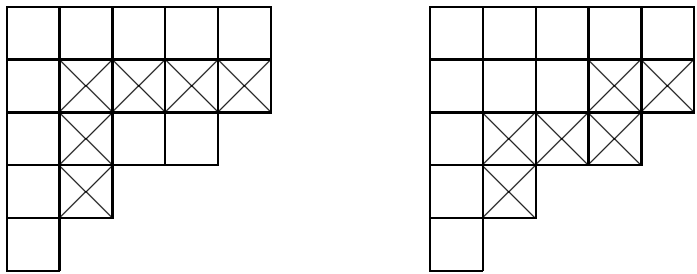
\begin{figure}[ht]

\begin{center}
\caption{A hook and its corresponding skew hook of leg length 2.} \vspace{.05in}
\end{center}

\begin{picture}(300,100)(-40,15)

\put(0,120){\line(1,0){100}}
\put(0,100){\line(1,0){100}}
\put(0,80){\line(1,0){100}}
\put(0,60){\line(1,0){80}}
\put(0,40){\line(1,0){40}}
\put(0,20){\line(1,0){20}}
\put(0,120){\line(0,-1){100}}
\put(20,120){\line(0,-1){100}}
\put(40,120){\line(0,-1){80}}
\put(60,120){\line(0,-1){60}}
\put(80,120){\line(0,-1){60}}
\put(100,120){\line(0,-1){40}}
\put(20,40){\line(1,1){20}}
\put(60,80){\line(1,1){20}}
\put(20,60){\line(1,1){20}}
\put(80,80){\line(1,1){20}}
\put(20,60){\line(1,-1){20}}
\put(20,100){\line(1,-1){20}}
\put(40,100){\line(1,-1){20}}
\put(40,100){\line(-1,-1){20}}
\put(60,100){\line(-1,-1){20}}
\put(60,100){\line(1,-1){20}}
\put(80,100){\line(1,-1){20}}
\put(20,80){\line(1,-1){20}}

\put(160,120){\line(1,0){100}}
\put(160,100){\line(1,0){100}}
\put(160,80){\line(1,0){100}}
\put(160,60){\line(1,0){80}}
\put(160,40){\line(1,0){40}}
\put(160,20){\line(1,0){20}}
\put(160,120){\line(0,-1){100}}
\put(180,120){\line(0,-1){100}}
\put(200,120){\line(0,-1){80}}
\put(220,120){\line(0,-1){60}}
\put(240,120){\line(0,-1){60}}
\put(260,120){\line(0,-1){40}}
\put(180,40){\line(1,1){60}}
\put(180,60){\line(1,1){20}}
\put(220,60){\line(1,1){40}}
\put(180,60){\line(1,-1){20}}
\put(200,80){\line(1,-1){20}}
\put(220,80){\line(1,-1){20}}
\put(220,100){\line(1,-1){20}}
\put(240,100){\line(1,-1){20}}
\put(180,80){\line(1,-1){20}}
\end{picture}

\end{figure}

\begin{theo}{\bf (Murnaghan-Nakayama Rule)}\label{thm-MN}
If $\lambda$ is a partition of $n$ and $\alpha\in S_n$ such that
$\alpha$ has cycle type $(\alpha_1,\alpha_2,\dots,\alpha_i)$, then we
have
\begin{eqnarray}\label{MN-rule}
\chi_{\lambda}(\alpha)=\sum_{\xi}(-1)^{ll(\xi)}\chi_{\lambda \backslash \xi}
(\alpha\backslash\alpha_1)
\end{eqnarray}
where the sum runs over all skew hooks $\xi$ of $\lambda$ having $\alpha_1$ cells and
$\chi_{\lambda\backslash\xi}(\alpha\backslash\alpha_1)$ denotes that character of the
representation $\lambda\backslash\xi$ evaluated at an element of cycle type 
$\alpha\backslash\alpha_1$.
\end{theo}

It is important to remark that when using the Murnaghan-Nakayama rule, if it is impossible
to remove a skew hook of the right size then the part of the sum corresponding to that
skew hook is zero. A good source for more information on the Murnaghan-Nakayama rule and
skew hooks is \cite{Sa}.

\begin{lem}\label{lem-MN-char}
For $m\in\mathbb{N}$ set $n=m(m+1)/2$.
Let $\lambda=(\lambda_1,\lambda_2,\dots,\lambda_m)\vdash n$ be a triangular Young
diagram such that $\lambda_i=m-i+1$.
Let $c_k$ be a cycle of length $k$.
If $k=4i+1$ for $i=1,2,\dots$ then
$\chi_{\lambda}(c_k)>0$. If $k$ is even then $\chi_{\lambda}(c_k)=0$.
\end{lem}

\begin{proof}
The  Murnaghan-Nakayama rule implies that
\begin{equation*}
\chi_{\lambda}(c_k)
=\sum_{|\xi|=k}(-1)^{ll(\xi)}d_{\lambda\backslash\xi}.
\end{equation*}

Any hook in $\lambda$ composed of
must have even leg length by construction so it
follows that $\chi_{\lambda}(c_k)>0$.
The second part of the proof follows directly from the
Murnaghan-Nakayama rule and the fact that
every hook in $\lambda$ will have odd hook length, making it impossible to
remove a skewhook of even length.
\end{proof}

Using the dimension and character estimates from (\ref{eq-square-dim})
and Lemma \ref{lem-MN-char}, one can replicate the ideas in the
proof of Theorem \ref{k-cont}. If $\mathbf c_k$ denotes here the conjugacy 
class of cycles of length $k$,  for any $\epsilon>0$ and
$t_n=\frac{n}{2}\log{n}$, we have
\begin{itemize}
\item[(1)] if $k_n$ is even then
\begin{equation*}
\lim_{n\ra\infty}d_2(h_{\mathbf c_{k_n},(1-\epsilon)t_n},u_n)=\infty
\end{equation*}
\item[(2)]if $k_n$ is odd and $k_n=4i_n+1$ for $i_n=1,2,3,\dots$ then
\begin{equation*}
\lim_{n\ra\infty}d_2(h_{\mathbf c_{k_n},(1-\epsilon)t_n},\overline{u}_n)=\infty.
\end{equation*}
\end{itemize}

\subsection{Total variation upper bounds in continuous time}

As we mentioned at the beginning of this section, the
mixing time of the continuous time process $h_{\mathcal{C}_n,t_n}$
will depend on whether one considers the total variation or the $\ell^2$ distance.
In this section we derive a total variation upper bound of type
$(n/\text{supp}(\mathcal{C}_n))\log{n}$ for the continuous time
process associated to $q_{\mathcal{C}_n}$.  In the next section, we shall show
that the $\ell^2$ mixing time has an upper bound of $(n/2)\log{n}$ for
the continuous time process when $\text{supp}(\mathcal{C}_n)\ra\infty$.

\begin{pro}
Let $\mathcal{C}_n$ be an even conjugacy class and
$\overline{u}_n$ to be the uniform measure on $A_n$.
Let $T_n$ be the total variation cutoff
time of $q_{\mathcal{C}_n}$ (in discrete time) and assume that $T_n\ra\infty$.
Then, for any $\epsilon>0$,
\begin{equation*}
\lim_{n\ra \infty}d_{\mbox{\tiny \em TV}}(h_{\mathcal{C}_n,(1+\epsilon)T_n},\overline{u}_n)=0
\end{equation*}
\end{pro}

\begin{proof}
Let
\begin{eqnarray*}
T^d_{n,\epsilon}
&=&\inf\{t\geq 0:d_{\text{\tiny TV}}(q_{\mathcal{C}_n}^{(t)},\overline{u}_n)\leq\epsilon\}\\
T^c_{n,\epsilon}
&=&\inf\{t\geq 0:d_{\text{\tiny{TV}}}(h_{\mathcal{C}_n,t},\overline{u}_n)\leq\epsilon\}.
\end{eqnarray*}
Corollary 4.1 in \cite{GY} shows that for any $\delta\in(0,1)$, $\epsilon>0$ and
$\eta\in(0,\epsilon)$ there exists an integer $N=N(\delta,\eta)$ such
that
\begin{equation*}
(1-\delta)T^c_{n,\epsilon}\leq T^d_{n,\eta}\;\;\text{for all $n\geq N$.}
\end{equation*}
In particular, for any $\epsilon>0$ we can find a $\delta\in(0,1)$ and an $N_1=N_1(\delta,\eta)$
such that for all $n\geq N$
\begin{equation*}
T^c_{n,\eta}\leq\sqrt{1+\epsilon}\;\;T^d_{n,\eta/2}.
\end{equation*}
From \cite{SP} we know that the random walk driven by $q_{\mathcal{C}_n}$ has cutoff, hence
for any $\epsilon>0$ and $\eta\geq 0$ there exists an $N_2=N_2(\epsilon,\eta)$ such that
for all $n\geq N_2$
\begin{equation*}
T_{n,\eta/2}^d\leq\sqrt{1+\epsilon}\;\; T_n.
\end{equation*}
Combining the inequalities above gives that for any $\epsilon>0$ and $\eta>0$ there exists an
$N=\max\{N_1,N_2\}$ such that for all $n\geq N$
\begin{equation*}
T_{n,\eta}^c\leq(1+\epsilon)T_n.
\end{equation*}
The desired result follows.

\end{proof}

\begin{rem}
In the case of the lazy random walk $\tilde{q}_{\mathcal{C}_n,1/2}$
defined in (\ref{q-lazy}), one can show that the total variation mixing time
is bounded by approximately twice that of the discrete time process $q_{\mathcal{C}_n}$.
(This is a more general phenomenon.)
We only treat the case when $\mathcal{C}_n$ is an even conjugacy class.
Note that
\begin{equation*}
d_{\mbox{\tiny TV}}(\tilde{q}_{\mathcal{C}_n}^{(t_n)},\overline{u}_n)
=\sum_{k=0}^t2^{-t_n}{ t_n\choose k}d_{\mbox{\tiny TV}}(q_{\mathcal{C}_n}^{(k)},\overline{u}_n).
\end{equation*}
For any constant $D>0$ set $\mathcal I_n=[0,t_n/2-D\sqrt{t_n}]\cup[t_n/2+D\sqrt{t_n},t_n]$.
Then we have that
\begin{equation*}
\sum_{k\in A_n}2^{-t_n}{t_n\choose k}d_{\mbox{\tiny TV}}(q_{\mathcal{C}_n}^{(k)},\overline{u}_n)
\leq \sum_{k\in A_n}2^{-t_n}{t_n\choose k}.
\end{equation*}
By the central limit theorem the right hand side tends to $0$ as $D$ tends to $\infty$.

Outside of the set $\mathcal I_n$ we get that
\begin{eqnarray*}
\sum_{k\notin A_n}2^{-t_n}{t_n\choose k}d_{\mbox{\tiny TV}}(q_{\mathcal{C}_n}^{(k)},\overline{u}_n)
&\leq& d_{\mbox{\tiny TV}}(q_{\mathcal{C}_n}^{(t_n/2-D\sqrt{t_n})},\overline{u}_n)
\sum_{k\notin A_n}2^{-t_n}{ t_n\choose k }\\
&\leq& d_{\mbox{\tiny TV}}(q_{\mathcal{C}_n}^{(t_n/2-D\sqrt{t_n})},\overline{u}_n).
\end{eqnarray*}

The arguments above shows that the
cutoff  time of the lazy walk is asymptotically $2T_n$.
A similar argument would show that for any $\epsilon>0$ the walk driven by
$\tilde{q}_{\mathcal{C}_n,\epsilon}$ has a cutoff  time asymptotically equal to 
$(1/(1-\epsilon))T_n$.
\end{rem}

\subsection{Continuous time $\ell^2$ upper bounds: $\mbox{supp}(\mathcal C_n)\ra\infty$}

Section \ref{sec-lower} shows that the $\ell^2$ mixing time
of $h_{\mathcal{C}_n,t}$ must be at least $(n/2)\log{n}$ for
all non trivial conjugacy classes.  We show that when $\text{supp}(\mathcal{C})$
goes to $\infty$ as $n\ra\infty$ and for the conjugacy class of 4-cycles
the continuous time random walk has an $\ell^2$ cutoff at $(n/2)\log{n}$.

\begin{theo}\label{clm-kcycle-grows}
Let $\mathcal{C}_n$ be a conjugacy class such that
$\text{\em supp}(\mathcal{C}_n)\ra\infty$ as $n\ra\infty$.
For any $\epsilon>0$,  and $t_n=(n/2)\log{n}$
\begin{itemize}
\item[(1)]
$\lim_{n\ra\infty}d_2\left(h_{\mathcal{C}_n,(1+\epsilon)t_n},u_n\right)=0$
if $\mathcal{C}_n$ is odd.
\item[(2)]
$\lim_{n\ra\infty}d_2\left(h_{\mathcal{C}_n,(1+\epsilon)t_n},\overline{u}_n\right)=0$
if $\mathcal{C}_n$ is even.
\end{itemize}
\end{theo}

\begin{proof} Let $\mathcal{C}_n$ be an odd conjugacy class.
Set $(\beta_i)_0^{n!-1}$ to  be the eigenvalues associated to the
measure $q_{\mathcal{C}_n}$ and $\lambda_i=1-\beta_i$. From
(\ref{dist-spect}) we know that
\begin{eqnarray*}
d_2(h_{\mathcal{C}_n,(1+\epsilon)t_n},u_n)^2
&=&\sum_{i=1}^{n!-1}e^{-2(1+\epsilon)t_n\lambda_i}\\
&=&\sum_{\lambda_i\leq 1-1/w}e^{-2t_n(1+\epsilon)\lambda_i}
 +\sum_{\lambda_i\geq 1-1/w}e^{-2t_n(1+\epsilon)\lambda_i}.
\end{eqnarray*}
We will use the following Calculus inequality.
\begin{claim}\label{eq-ineq-w}
For $w\geq 4$ and $0\leq x\leq 1-1/w$ we have that $2\log(1-x)\geq
-wx$.
\end{claim}

For $1/3\geq \epsilon>0$, $w=(1+\epsilon)/\epsilon\geq 4$, so by
the claim above and (\ref{dist-spect}) we obtain
\begin{eqnarray*}
d_2(h_{\mathcal{C}_n,(1+\epsilon)t_n},u_n)^2
&\leq& \sum_{1/w\leq \beta_i} \beta_i^{\epsilon 4t_n}+n!e^{-2t_n(1-1/w)(1+\epsilon)}\\
&=&\sum_{1/w\leq \beta_i} \beta_i^{\epsilon 4t_n}+n!e^{-n\log{n}}
\end{eqnarray*}

We know  that
the eigenvalues of $q_{\mathcal{C}_n}$ are just the normalized
characters $\chi_{\rho}(c_n)/d_{\rho}$, $c_n\in \mathcal C_n$,
that occur with multiplicity $d_{\rho}^2$.
Let $\rho_1$ and $\rho_2$ be the trivial and sign representations
respectively.  When $\mathcal{C}_n$ is odd $\chi_{\rho_2}(c_n)/d_{\rho_2}=-1$,
so the character associated to the sign representation does not contribute to the
sum of eigenvalues above.
Furthermore, (see, e.g., \cite[Lemma 2]{MSP})
\begin{equation}\label{dist-k-discrete}
d_2(q_{\mathcal{C}_n}^{(2t)},\overline{u}_n)^2
=\frac{1}{2}\sum_{\rho\neq\rho_1,\rho_2}
d_{\rho}^2\left(\frac{\chi_\rho(c_n)}{d_{\rho}}\right)^{4t}.
\end{equation}
It now follows that
\begin{eqnarray*}
d_2(h_{\mathcal{C}_n,(1+\epsilon)t_n},u_n)^2
&\leq&\sum_{\rho\neq\rho_1,\rho_2}
d_{\rho}^2\left(\frac{\chi_{\rho}(c_n)}{d_{\rho}}\right)^{\epsilon 4t_n}
+n!e^{-n\log{n}}\\
&\leq&2d_2(q^{(\epsilon 2t_n)}_{\mathcal{C}_n},\overline{u}_n)^2+n!e^{-n\log{n}}.
\end{eqnarray*}

In \cite{MSP} it is shown that there exists a fixed constant
$D>0$   such that for $t_n$ even and
$t_n\geq (Dn/\text{supp}(\mathcal{C}_n))\log{n}$ then
$d_2(q_{\mathcal{C}_n}^{(t_n)},\overline{u}_n)\ra 0$ as $n\ra\infty$.
Since $\text{supp}(\mathcal{C}_n)\ra\infty$ as
$n\ra\infty$ then for large enough $n$ we have that
$\epsilon n\log{n}\geq(Dn/\text{supp}(\mathcal{C}_n))\log{n}$ and the desired 
result follows.  The case when $\mathcal{C}_n$ is an even conjugacy class
can be treated in a similar way.
\end{proof}

\begin{rem}
Let $\tilde{q}_{\mathcal{C}_n}$ be the lazy chain  defined in
(\ref{q-lazy}).  In the remark after Theorem \ref{k-cont-odd} 
it is noted that the random walk driven by $\tilde{q}_{\mathcal{C}_n}$ will 
have a $\ell^2$ lower bound on the mixing time of $(n/2)\log_{2}(n)$.  
A matching upper bound for conjugacy classes $\mathcal{C}_n$
such that $\text{supp}(\mathcal{C}_n)\ra\infty$ as $n\ra\infty$ follows 
from an argument similar to the proof of Theorem \ref{clm-kcycle-grows}.
\end{rem}

\subsection{$\ell^2$ continuous time upper bound: $4$-cycles.}
The next theorem gives a sharp $\ell^2$  upper bound for the $4$-cycle walk.
In the case when $\mbox{supp}(\mathcal C_n)\ra \infty$ we relied on the
(rather deep) results of \cite{MSP,SP,Roi} concerning the discrete time case
to obtain a continuous time result matching our lower bound. This technique 
does not work for conjugacy classes with fixed support size. We conjecture
that, with out any restriction on $\mbox{supp}(\mathcal C_n)$, $(n/2)\log n$
is a $\ell^2$ cutoff time for the family  $(h_{\mathcal C_n,t})$. 
Note however that there is no reasons to hope for a proof simpler than
that for random transposition. In discrete time, the only cases with fixed
where support size for which  the $\ell^2$ cutoff time has been determined
are the cases of support size at most $6$ (and the $7$-cycles) treated in \cite{Ro,Ro1}. 
Using the techniques of \cite{Ro,Ro1} one can probably
treat the corresponding continuous time processes, but this will be hard work. Here 
we focus on the $4$-cycle walk. The reason is that
we are able to reduce most technical computations to those already done above for 
transposition. We note that is is unlikely such reduction would work easily for 
$3$-cycles and other even conjugacy classes (see \cite{Ro,Ro1}).

Recall that the conjugacy class of $4$-cycles is denoted by $\mathbf c_4$.
We let $c_4$ be a given $4$-cycle.
\begin{theo}\label{dist-4-cycle}
For $n\geq 11$, $c\geq 2$ and $t\geq(n/2)(\log{n}+c)$
\begin{equation*}
d_2\left(h_{\mathbf c _4,t},u_n\right)\leq e^{-(c-2)}
\end{equation*}
\end{theo}
We will use (\ref{main=}) again and bound $\chi_{\rho}(c_4)/d_{\rho}$,
$c_4\in \mathbf c_4$,
with the same upper bounds that we used for
 $\chi_{\rho}(\tau)/d_{\rho}$, $\tau\in \mathbf c_2$,
in the case of  transpositions
in Proposition \ref{pro-cont-tr}.   In order to do this we will need
the following definitions and lemmas

\begin{defin}
If $\lambda'=(\lambda_1',\dots,\lambda_j')$ and
$\lambda=(\lambda_1,\dots,\lambda_k)$ are two Young diagrams such that
$\sum_{i=1}^j\lambda_i'=\sum_{i=1}^k\lambda_i=n$ and it is possible to
get from $\lambda$ to $\lambda'$ by moving boxes up to the right then
we say that $\lambda'\geq\lambda.$
\end{defin}

\begin{defin}
Let $\lambda=(\lambda_1,\lambda_2,\dots,\lambda_m)$ denotes a Young diagram
such that $\sum_{i=1}^m\lambda_i=n$.  For any integer $l\geq 0$
\begin{equation*}
M_{\lambda,2l}
=\sum_{j=1}^m\left\{(\lambda_j-j)^l(\lambda_j-j+1)^l-j^l(j-1)^l\right\}.
\end{equation*}
\end{defin}

\begin{lem}\label{eq-M-ineq}
Let $\lambda'$ and $\lambda$ be two Young diagrams associated to irreducible representations
of $S_n$. If $\lambda'\geq \lambda$ then $M_{\lambda',2l}\geq M_{\lambda,2l}$ for all $l\geq 0$.
\end{lem}

\begin{proof} It suffices to show that $M_{\lambda',2l}\geq M_{\lambda,2l}$ for that case
when $a<b$ and $\lambda'_a=\lambda_a+1$, $\lambda_b'=\lambda_b-1$ and $\lambda'_c=\lambda_c$ for
$c\neq a,b$.  In this case,
\begin{eqnarray*}
M_{\lambda',2l}-M_{\lambda,2l}&=&
(\lambda_a-a+1)^l\left\{\left((\lambda_a-a+1)+1\right)^l-\left((\lambda_a-a+1)-1\right)^l\right\}\\
&&+(\lambda_b-b)^l\left\{(\lambda_b-b-1)^l-(\lambda_b-b+1)^l\right\}.
\end{eqnarray*}

\noindent Set $x=\lambda_a-a+1$ and $y=\lambda_b-b$ then $n\geq x\geq y\geq 1-n$ and
$M_{\lambda',2l}-M_{\lambda,2l}=f_{x,y}(l)$ where
\begin{equation*}
f_{x,y}(l)=x^l\left\{(x+1)^l-(x-1)^l\right\}+y^l\left\{(y-1)^l-(y+1)^l\right\}.
\end{equation*}

\noindent In \cite{Dia} Diaconis shows that $f_{x,y}(1)\geq 0$ for $n\geq x\geq y\geq 1-n$
which implies that $M_{\lambda',2}\geq M_{\lambda,2}$.  We will show the general case by
induction.  Assume that $f_{x,y}(l)\geq 0$ then
\begin{eqnarray*}
f_{x,y}(l+1)&=&
x^{l+1}\left\{(x+1)^{l+1}-(x-1)^{l+1}\right\}+y^{l+1}\left\{(y-1)^{l+1}-(y+1)^{l+1}\right\}\\
&=&x^2\left(x^l\left\{(x+1)^l-(x-1)^l\right\}\right)+y^2\left(y^l\left\{(y-1)^l-(y+1)^l\right\}\right)\\
&&+x^{l+1}\left\{(x+1)^l+(x-1)^l\right\}-y^{l+1}\left\{(y+1)^l+(y-1)^l\right\}\\
&\geq& x^{l+1}\left\{(x+1)^l+(x-1)^l\right\}-y^{l+1}\left\{(y+1)^l+(y-1)^l\right\}.
\end{eqnarray*}

\noindent The last inequality follows since $f_{x,y}(l)\geq 0$.
To conclude that $f_{x,y}(l+1)\geq 0$ we must check the following three cases.

\bigskip

\noindent{\bf Case 1:} $x\geq y\geq 0$ .  
This case follows directly from the assumption $x\geq y$.

\smallskip

\noindent{\bf Case 2:} $x\geq 0 $ and $y\leq 0$.  Note that in this case
\begin{eqnarray*}
x^{l+1}\left\{(x+1)^l+(x-1)^l\right\}\geq 0\\
y^{l+1}\left\{(y+1)^l+(y-1)^l\right\}\leq 0.
\end{eqnarray*}
\noindent The last inequality follows from the fact that $l$ and $l+1$
are an odd and even numbers.

\smallskip

\noindent{\bf Case 3:} $y\leq x\leq 0$.  In this case let $\tilde{x}=-x$ and $\tilde{y}=-y$ then
$\tilde{y}\geq\tilde{x}\geq 0$ and
\begin{eqnarray*}
&&x^{l+1}\left\{(x+1)^l+(x-1)^l\right\}-y^{l+1}\left\{(y+1)^l+(y-1)^l\right\}=\\
&&\tilde{y}^{l+1}\left\{(\tilde{y}+1)^l+(\tilde{y}-1)^l\right\}-
\tilde{x}^{l+1}\left\{(\tilde{x}+1)^l+(\tilde{x}-1)^l\right\}.
\end{eqnarray*}
\noindent Case 3 now follows directly from Case 1.
\end{proof}

\begin{lem}\label{eq-M}
Let $\lambda=(\lambda_1,\lambda_2,\dots,\lambda_j)$ denote a Young diagram such
that $\sum \lambda_i=n$. Then
\begin{equation*}
M_{2l,\lambda}\leq n(\lambda_1-1)^l\lambda_1^{l-1}.
\end{equation*}
\end{lem}

\begin{proof}
\begin{eqnarray*}
M_{2l,\lambda}&=&\sum_{i=1}^j(\lambda_j-j)^l(\lambda_j-j+1)^l-j^l(j-1)^l\\
&\leq&\sum_{\lambda_j\geq j-1}(\lambda_j-j)^l(\lambda_j-j+1)^l\\
&&+
      \sum_{\lambda_j<j-1}(\lambda_j-j)^l(\lambda_j-j+1)^l-j^l(j-1)^l
\end{eqnarray*}
For $0\leq \lambda_j\leq j-1$ it is true that
$|\lambda_j-j|\leq j$ and $|\lambda_j-j+1|\leq j-1$ which
implies that the second sum in the inequality above is negative.
Therefore

\begin{equation*}
M_{2l,\lambda}
\leq\sum_{\lambda_j\geq j-1}(\lambda_j-j)^l(\lambda_j-j+1)^l
\leq n(\lambda_1-1)^l\lambda_1^{l-1}.
\end{equation*}
\end{proof}

\begin{lem}\label{lem-r4}
\noindent Let $\rho$ be an irreducible representation of $S_n$ and $\lambda$
the associated Young diagram.  For $n\geq 11$ the normalized character
$r_4(\lambda)=\chi_\rho(c_4)/d_{\rho}$ can be bounded as follows.

\begin{eqnarray*}
r_4(\lambda)\leq\left\{
\begin{array}{ll}
1-\frac{2\lambda_1(n-\lambda_1)}{n(n-1)}& \text{if $\lambda_1\geq n/2$}\\
\frac{\lambda_1-1}{n-1}& \text{if $\lambda_1\leq n/2$.}
\end{array}\right.
\end{eqnarray*}
\end{lem}

\begin{proof}
Set $\lambda=(\lambda_1,\lambda_2,\dots,\lambda_j)$.
In \cite{Ing}, Ingram shows that
\begin{equation}\label{eq-M-4}
\frac{n!}{(n-4)!}r_4(\lambda)=M_{4,\lambda}-2(2n-3)M_{2,\lambda}.\\
\end{equation}

\noindent Lemma \ref{eq-M-ineq} implies that $M_{2,\lambda}\geq M_{2,\lambda '}$
where $\lambda '=(\lambda_1, 1, 1,\dots, 1)$.  We get,
\begin{eqnarray*}
M_{2,\lambda'}&=&(\lambda_1-1)\lambda_1+\sum_{j=2}^{n-\lambda_1}(1-j)(2-j)-j(j-1)\\
&=&(\lambda_1-1)\lambda_1-2\sum_{j=1}^{n-\lambda_1-1}j\\
&=&(\lambda_1-1)\lambda_1-(n-\lambda_1-1)(n-\lambda_1).
\end{eqnarray*}
If $\lambda_1\geq n/2$ then $M_{4,\lambda}\leq M_{4,(\lambda_1,n-\lambda_1)}$.  Note that
\begin{eqnarray*}\lefteqn{
M_{4,(\lambda_1,n-\lambda_1)}=
(\lambda_1-1)^2\lambda_1^2+(n-\lambda_1-1)^2(n-\lambda_1)^2-4}&&\\
&\leq& (\lambda_1-1)^2\lambda_1^2+(n-\lambda_1-1)^2(n-\lambda_1)^2\\
&=&[(\lambda_1-1)\lambda_1-(n-\lambda_1-1)(n-\lambda_1)]^2
+2(\lambda_1-1)\lambda_1(n-\lambda_1-1)(n-\lambda_1).
\end{eqnarray*}
\noindent
Hence if $\lambda_1\geq n/2$, we have
\begin{eqnarray*}
M_{4,\lambda}-2(2n-3)M_{2,\lambda} &=&
[(\lambda_1-1)\lambda_1-(n-\lambda_1)(n-\lambda_1-1)]\\
&&\times [(\lambda_1-1)\lambda_1-(n-\lambda_1)(n-\lambda_1-1)-2(2n-3)]\\
&& +2(\lambda_1-1)\lambda_1(n-\lambda_1)(n-\lambda_1-1).
\end{eqnarray*}
\noindent
Note that
\begin{equation*}
(\lambda_1-1)\lambda_1-(n-\lambda_1)(n-\lambda_1-1)-2(2n-3)\leq (n-2)(n-3).
\end{equation*}
\noindent It follows that
\begin{eqnarray}
M_{4,\lambda}-2(2n-3)M_{2,\lambda}
&\leq&(n-2)(n-3)[(\lambda_1-1)\lambda_1-(n-\lambda_1)(n-\lambda_1-1)]\nonumber\\
&&+2\lambda_1(\lambda_1-1)(n-\lambda_1)(n-\lambda_1-1).\label{eq-M4}
\end{eqnarray}

\noindent 
If $\lambda_1\geq n-1$ then $2\lambda_1(\lambda_1-1)(n-\lambda_1)(n-\lambda_1-1)=0$.
If $\lambda_1\leq n-2$ then  $\lambda_1(\lambda_1-1)\leq (n-2)(n-3)$.
In either case, (\ref{eq-M4}) gives that 
\begin{eqnarray*}
r_4(\lambda)
&\leq&\frac{(\lambda_1-1)\lambda_1-(n-\lambda_1)(n-\lambda_1-1)}{n(n-1)}
+\frac{2(n-\lambda_1)(n-\lambda_1-1)}{n(n-1)}\\
&=&1-\frac{2\lambda_1(n-\lambda_1)}{n(n-1)}.
\end{eqnarray*}

\noindent Next, we show the second part of the inequality.  By Lemma 
\ref{eq-M} and
(\ref{eq-M-4}) we have that for $\lambda_1\leq n/2$
\begin{eqnarray*}
\left|r_4(\lambda)\right|
&\leq&\frac{(n-4)!}{n!}\left[n(\lambda_1-1)^2\lambda_1+2(2n-3)n(\lambda_1-1)\right]\\
&=&\left(\frac{\lambda_1-1}{n-1}\right)
   \left[\frac{(\lambda_1-1)\lambda_1+2(2n-3)}{(n-2)(n-3)}\right]\\
&\leq&\left(\frac{\lambda_1-1}{n-1}\right)\left[\frac{n^2/4+4n-6}{(n-2)(n-3)}\right]
\leq \frac{\lambda_1-1}{n-1}
\end{eqnarray*}

\noindent The last inequality holds for $n\geq 11$.
\end{proof}

\begin{proof}[Proof of Theorem \ref{dist-4-cycle}.]
Recall that
$$d_2(h_{\mathbf c_4,t},u_n)^2
=\sum_{\lambda\neq 1}d_{\lambda}^2
\exp\left\{-2t\left(1-r_4(\lambda)\right)\right\}.$$
In order to obtain the desired $e^{-2(c-2)}$ constant we will bound
the term corresponding to $\lambda=(n-1,1)$ separately. For $\lambda=(n-1,1)$ 
we get that 
\begin{equation*}
M_{2,(n-1,1)}=(n-2)(n-1)-2
\;\;\text{and}\;\;
M_{4,(n-1,1)}=(n-2)^2(n-1)^2-4
\end{equation*}
which implies $r_4((n-1,1))=1-4/(n-1)$.  So for $t\geq(n/2)(\log{n}+c)$,
\begin{equation*}
d_{(n-1,1)}^2\exp\{-2t(1-r_{(n-1,1)}(4))\}\leq (n-1)^2\exp\{-4(\log{n}+c)\}\leq e^{-4c}/n^2.
\end{equation*}

\noindent 
Lemma \ref{lem-r4} and equation (\ref{dim})
imply that for $t\geq(n/2)(\log{n}+c)$ we have
$d_2(h_{t,4},u_n)^2\leq e^{-4c}/n^2+S_1 +S_2$ where
\begin{eqnarray*}
S_1&=&\sum_{j=2}^{n/2}\left(\frac{n!}{(n-j)!}\right)^2\left(\frac{1}{j!}\right)
\exp\left\{-(\log{n}+c)\left(\frac{2j(n-j)}{n-1}\right)\right\}\\
S_2&=&\sum_{j=n/2}^{n-1}\left(\frac{n!}{(n-j)!}\right)^2\left(\frac{1}{j!}\right)
\exp\left\{-j(\log{n}+c)\right\}.
\end{eqnarray*}

\noindent For a more detailed description on how to obtain the sums $S_1$ and $S_2$
see the proof of Proposition \ref{rt}. For $c\geq 2$ we have that 
\begin{itemize}
\item[(1)]$-c(2j)(n-j)/(n-1)\leq -2(c-2)-2j$ when $2\leq j\leq n/2$ and
\item[(2)]$-jc\leq -2(c-2)-2j$ when $j\geq 2$.
\end{itemize}
It follows that
\begin{equation*}
d_2(h_{t,4},u_n)^2\leq e^{-2(c-2)}\left(\frac{1}{n^2}+\sum_{j=1}^{n/2}A_j+\sum_{j=n/2}^nB_j\right)
\end{equation*}
where $A_j$ and $B_j$ are defined in equations (\ref{Aj-cont}) and (\ref{Bj-cont}).
Lemmas \ref{lem-Aj-cont} and \ref{lem-Bj-cont} now imply that for $n\geq 11$
\begin{equation*}
d_2(h_{t,4},u_n)^2\leq e^{-2(c-2)}
\left(\frac{1}{n^2}+\frac{2}{3}+\frac{1}{4}+2\left(\frac{2}{e}\right)^{3n/2}\right)
\leq e^{-2(c-2)}.
\end{equation*}
\end{proof}

%%%%%%%%%%%%%%%%%%%%%%%%%%%%%%%%%%%%%%%%%%%%%%%%%%%%%%%%%%%
%%%%%%%%%%%%%%%%%%%%%%%%%%%%%%%%%%%%%%%%%%%%%%%%%%%%%%%%%%%


\begin{thebibliography}{00}


\bibitem{Al} 
Aldous, D. {\em
 Random walks on finite groups and rapidly mixing Markov chains.}
 Seminar on probability, XVII,  243--297, Lecture Notes in Math., 986, 
Springer, 1983.


\bibitem{AD} Aldous, D.  and Diaconis, P.
{\em  Shuffling cards and stopping times.}
Amer. Math. Monthly  93, (1986)  no. 5, 333--348,

\bibitem{AF} Aldous, D. and Fill, J.
{\em Reversible Markov Chains and Random Walks on Graphs.}
http://www.stat.berkeley.edu/users/aldous/

\bibitem{DFP} Diaconis, P, Fill, J. and  Pitman, J.
{\em Analysis of top to random shuffles.}
Combin. Probab. Comput. 1 (1992), no. 2, 135--155.

\bibitem{GY} Chen, G-Y.
{\em The cutoff phenomenon for finite Markov chains.}
Ph.D. Dissertation, Cornell University, 2006.

\bibitem{GYSC} Chen, G-Y  and Saloff-Coste, L.
{\em The cutoff phenomenon for ergodic Markov processes}.
 	Electronic Journal of Probability 13,  (2008), 26--78.
  	
\bibitem{Dia} Diaconis, P. {\em
Group representations in probability and statistics.}
Institute of Mathematical Statistics
Lecture Notes---Monograph Series, 11.
Institute of Mathematical Statistics, Hayward, CA, 1988.

\bibitem{Dia2} Diaconis, P.
{\em Finite Fourier Methods: Access to Tools}.
Proceedings of Symposia in Applied Mathematics, Vol. 44,
American Mathematical  Society, Providence, RI, 1991.


\bibitem{Dcut}
Diaconis, P. {\em  
The cutoff phenomenon in finite Markov chains.}  
Proc. Nat. Acad. Sci. U.S.A.  93  (1996),  1659--1664. 

\bibitem{Dtran} Diaconis, P. {\em 
Random walks on groups: characters and geometry.}  
Groups St. Andrews 2001 in Oxford. Vol. I,  120--142, 
London Math. Soc. Lecture Note Ser., 304, Cambridge Univ. Press, 
2003.


\bibitem{DS} Diaconis, P. and Shahshahani, M.
{\em Generating a random permutation with random transpositions.}
 Z. Wahrsch. Verw. Geb {\bf 52}, no. 2, 159--179, 1981.

\bibitem{DS-Comparison}Diaconis, P. and Saloff-Coste, L.
{\em  Comparison techniques for random walk on finite groups.}
Ann. Probab.  21  (1993),  no. 4, 2131--2156.

\bibitem{DSC}Diaconis, P. and Saloff-Coste, L.
{\em Random walks on finite groups: a survey of analytic techniques}
In Probability measures on groups and related structures XI
(Oberwolfach 1994), 44-75.  World Scientific.

\bibitem{Fe} Feller, W.
{\em An introduction to probability theory and its applications.} Vol. I.
Third edition, John Wiley \& Sons, Inc., New York-London-Sydney, 1968.

\bibitem{FOW} Flatto, L., Odlyzko A. M. and Wales D. B.
{\em Random Shuffles and Group Representations}
The Annals of Probability, Vol. 13, No. 1 (Feb., 1985), 154-178.

\bibitem{HJ} Horn, R and Johnson, C. {\em  Topics in matrix analysis.}
Cambridge; New York: Cambridge University Press, 1991.

\bibitem{Ing}Ingram, R. E. {\em Some characters of the symmetric group.}
Proc. Amer. Math. Soc. 1, (1950). 358--369.

\bibitem{Jam}James, G.
{\em The representation theory of the symmetric group.}
Lecture Notes in Mathematics, 682, Springer-Verlag, New York.

\bibitem{JK}
James, G.and  Kerber, A.
{\em The representation theory of the symmetric group.}
Encyclopedia of Mathematics and its Applications, 16.
Addison-Wesley Publishing Co., Reading, Mass., 1981.

\bibitem{Knapp}
Knapp, A.
{\em Lie groups beyond an introduction.}
Progress in Mathematics, 140.
Birkhauser Boston, Inc., Boston, MA, 1996.


\bibitem{MSP}
M\"uller, T. and  Schlage-Puchta, J.
{\em Character theory of symmetric groups,
subgroup growth of Fuchsian groups, and random walks.}
Adv. Math.  213  (2007),  919--982.

\bibitem{LP}Lulov, N. and  Pak, I.
{\em Rapidly mixing random walks and bounds on characters of the symmetric group.}  
J. Algebraic Combin.  16  (2002),  151--163.


\bibitem{RS}
Rattan, A. and \'Sniady, P.
{\em Upper bound on the characters of the symmetric goups for balanced Young diagrams
and a generalized Frobenius Formula}
Advances in Mathematics 218 (2008) 673-695.

\bibitem{Roi} Roichman, Y.  {\em Upper bounds on the characters of the symmetric group}
Invent. Math. 125 (1996), no. 3, 451-485.

\bibitem{Ro} Roussel, S. {\em Marches al\'eatoires sur le group sym\'etrique.}
Ph.D. thesis, Universit\'e Paul Sabatier, Toulouse III, 1999.

\bibitem{Ro1} Roussel, S. {\em Ph\'enom\'ene de cutoff pour certaines marches al\'eatoires
sur le groupe sy\'emetrique.} Colloquium Mathematicum (2000), no. 1, vol. 86, 111-135.

\bibitem{Sa} Sagan, B.E. {\em The Symmetric Group, Representations, Combinatorial Algorithms,
and Symmetric Functions.} Springer-Verlag New York, Inc., 2001.

\bibitem{SCMZ} Saloff-Coste, L. {\em Precise estimates on the rate at 
which certain diffusions tend to equilibrium.}  Math. Z.  217  (1994),  no. 4, 641--677.

\bibitem{StF} Saloff-Coste, L. {\em Lectures on finite Markov chains.}
Lectures on probability theory and statistics (Saint-Flour, 1996),  301--413,
Lecture Notes in Math., 1665, Springer, Berlin, 1997.

\bibitem{S-K} Saloff-Coste, L. {\em Random walks on finite groups.}
In Probability on discrete structures (H. Kesten, Ed.), 261-346,
Encyclopedia Math. Sci., 110, Springer, Berlin, 2004.

\bibitem{SC} Saloff-Coste, L. {\em Total variation lower bounds for finite
Markov chains: Wilson's lemma.}  Random walks and geometry,
515--532, Walter de Gruyter GmbH \& Co. KG, Berlin, 2004.

\bibitem{SCZ} Saloff-Coste, L. and Z\'u\~niga, J. {\em
Convergence of some time inhomogeneous Markov chains via spectral techniques.}  
Stochastic Process. Appl.  117  (2007),  961--979

\bibitem{SP} Schlage-Puchta, J.
{\em Mixing Properties of Finite Permutation Groups}
In preparation.

\bibitem{UR} Uyemura-Reyes, J.
{\em Random Walk, Semi-direct Products, and Card Shuffling.}
Ph.D. thesis, Stanford University, 2002.

\bibitem{Wil}Wilson, D.
{\em Mixing times of lozenge tiling and card shuffling Markov chains.}
Ann. Appl. Prob. Volume 14, Number 1 (2004), 274-325.
\end{thebibliography}
\end{document}